\documentclass[12pt]{article}
\usepackage{amsmath,amsfonts,amssymb,amsthm}

\def\noi{\noindent}

\def\pf{\noi{\bf Proof.\ \,}}
\def\eop{{$\square$}}
\def\qed{{$\square$}}

\def\em{\it }        

\def\labtt#1{\label {#1}  }
\def\labttr#1{\label {#1} \rm }

\def\refpp#1{(\ref {#1})}

\def\a{\alpha}
\def\b{\beta}

\def\o{\omega}

\def\th{\theta}
\def\vep{\varepsilon}

\def\L{\Lambda}

\def\CC{{\mathbb C}}
\def\FF{{\mathbb F}}
\def\QQ{{\mathbb Q}}

\def\ZZ{{\mathbb Z}}

\def\la{\langle}
\def\ra{\rangle}
\def\<{\langle}
\def\>{\rangle}

\def\bs{\it}            
\def\Aut{{\bs Aut}}

\def\ker{{\bs ker}}

\def\det{{\bs det}}

\def\half{{1 \over 2}}

\def\dual#1{#1^*}        

\def\dg#1{{\cal D}({#1})}

\def\brw#1{BRW^+(2^{#1})}

\def\cvcch{cvcc$\half$ }

\newcommand{\MM}{\mathbb{M}}

\def\bsays#1{{\bf *** BOB SAYS: #1 ***}}
\def\bob#1{{\bf *** BOB SAYS: #1 ***}}

\begin{document}

\newtheorem{mthm}{Main Theorem}
\newtheorem{thm}{Theorem}[section]
\newtheorem{prop}[thm]{Proposition}
\newtheorem{lem}[thm]{Lemma}
\newtheorem{rem}[thm]{Remark}
\newtheorem{coro}[thm]{Corollary}
\newtheorem{conj}[thm]{Conjecture}
\newtheorem{de}[thm]{Definition}
\newtheorem{hyp}[thm]{Hypothesis}

\newtheorem{nota}[thm]{Notation}
\newtheorem{ex}[thm]{Example}
\newtheorem{proc}[thm]{Procedure}



\newcommand{\dih}[2]{DIH_{#1}(#2)}
\newcommand{\EE}{EE_8}
\newcommand{\TA}{\mathcal{T\!\!A}}
\newcommand{\SA}{\mathcal{S\!\!A}}
\def\bsays#1{{\bf *** BOB SAYS: #1 ***}}
\def\bob#1{{\bf *** BOB SAYS: #1 ***}}

\def\chsays#1{{\bf *** Ching Hung SAYS: #1 ***}}

%

\begin{center}
{\Large \bf
Moonshine paths for $3A$  and $6A$ nodes of the extended $E_8$-diagram}
\end{center}

\begin{center}
\vspace{10mm}
Robert L.~Griess Jr.
\\[0pt]
Department of Mathematics\\[0pt] University of Michigan\\[0pt]
Ann Arbor, MI 48109  USA  \\[0pt]
{\tt rlg@umich.edu}\\[0pt]
\vskip 1cm

Ching Hung Lam
\\[0pt]
Institute of Mathematics \\[0pt]
Academia Sinica\\[0pt]
Taipei 10617, Taiwan\\[0pt]
{\tt chlam@math.sinica.edu.tw}\\[0pt]
\vskip 1cm

\end{center}

\begin{abstract}
We continue the program, begun in \cite{gl3cpath}, to make a moonshine path
between a node of the extended $E_8$-diagram and the Monster.
Our theory is a concrete model expressing some of the mysterious connections identified by John McKay, George Glauberman and Simon Norton.
In this article, we treat the
$3A$ and $6A$-nodes.    We determine the orbits of triples $(x,y,z)$ in the Monster where $z\in 2B$, $x, y \in 2A \cap C(z)$ and $xy\in 3A \cup 6A$.   Such $x, y$ correspond to a rootless $EE_8$-pair in the Leech lattice.
For the $3A$ and $6A$ cases, we shall say something about the ``half Weyl groups", which are proposed in  the Glauberman-Norton theory.  Most work in this article is with lattices, due to their connection with dihedral subgroups of the Monster.
These lattices are $M+N$, where $M, N$ is the relevant pair of $EE_8$-sublattices, and their annihilators in the Leech lattice.   The isometry groups of these four lattices are analyzed.
\end{abstract}

\newpage
\tableofcontents

\medskip

\section{Introduction}

Moonshine path theory \cite{gl3cpath}
is intended to understand the discoveries of McKay \cite{Mc}
and Glauberman-Norton \cite{gn} which connect
 the extended $E_8$-diagram  and the Monster, denoted $\MM$,
and place these relationships in a broader mathematical context.
The paths involve series of small steps,
each using established mathematical theories.
The introduction of  \cite{gl3cpath} has a detailed discussion of context, which involves lattices, vertex operator algebras (VOAs), Lie theory and finite groups.
In \cite{gl3cpath} and \cite{gl5A},  we treated the
cases of the $3C$-node and the $5A$-node in detail.
The present article treats the cases of the $3A$-node and the $6A$-node.

Let us first review the background and the main ideas in \cite{gl3cpath}.
It is well known (cf. \cite{c1}) that $2A$-involutions of
the Monster simple group $\MM$ satisfy a 6-transposition property, that
is, given a pair of $2A$-involutions $(x,y)$ in $\MM$, the product $xy$ has order $\leq 6$. John McKay \cite{Mc} noticed a remarkable
correspondence with  the extended $E_8$-diagram $\tilde{E_8}$ as follows.

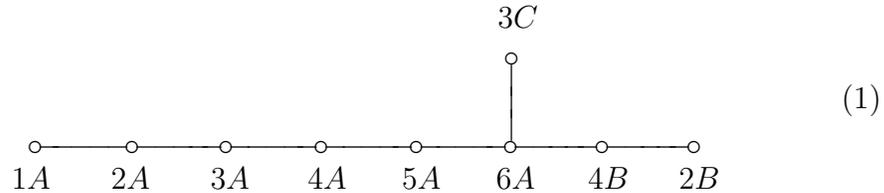
\begin{equation}\label{McKay diagram}
\begin{array}{l}
  \hspace{184pt} 3C\\
  \hspace{186.2pt}\circ \vspace{-6.2pt}\\
   \hspace{187.5pt}| \vspace{-6.1pt}\\
 \hspace{187.5pt}| \vspace{-6.1pt}\\
 \hspace{187.5pt}| \vspace{-6.2pt}\\
  \hspace{6pt} \circ\hspace{-5pt}-\hspace{-7pt}-\hspace{-5pt}-
  \hspace{-5pt}-\hspace{-5pt}-\hspace{-5pt}\circ\hspace{-5pt}-
  \hspace{-5pt}-\hspace{-5pt}-\hspace{-6pt}-\hspace{-7pt}-\hspace{-5pt}
  \circ \hspace{-5pt}-\hspace{-5.5pt}-\hspace{-5pt}-\hspace{-5pt}-
  \hspace{-7pt}-\hspace{-5pt}\circ\hspace{-5pt}-\hspace{-5.5pt}-
  \hspace{-5pt}-\hspace{-5pt}-\hspace{-7pt}-\hspace{-5pt}\circ
  \hspace{-5pt}-\hspace{-6pt}-\hspace{-5pt}-\hspace{-5pt}-
  \hspace{-7pt}-\hspace{-5pt}\circ\hspace{-5pt}-\hspace{-5pt}-
  \hspace{-6pt}-\hspace{-6pt}-\hspace{-7pt}-\hspace{-5pt}\circ
  \hspace{-5pt}-\hspace{-5pt}-
  \hspace{-6pt}-\hspace{-6pt}-\hspace{-7pt}-\hspace{-5pt}\circ
  \vspace{-6.2pt}\\
  \vspace{-10pt} \\
  1A\hspace{23pt} 2A\hspace{23  pt} 3A\hspace{22pt} 4A\hspace{21pt}
  5A\hspace{21pt} 6A\hspace{20pt} 4B\hspace{19pt} 2B\\
\end{array}
\end{equation}
There are 9 conjugacy classes of such pairs $(x,y)$, and the orders of the 8
products $|xy|$, for $x\ne y$,  are the coefficients of the highest root in the
$E_8$-root system.  The 9 nodes are labeled with 9 conjugacy classes of
$\mathbb M$ containing the $xy$.

In 2001, George Glauberman and Simon Norton \cite{gn} enriched this theory by
adding details about the centralizers in the Monster of such pairs of
involutions and relations involving the associated modular forms. Let $(x,y)$
be such a pair and let $n(x,y)$ be its associated node.  Let $n'(x,y)$ be  the
subgraph of $\tilde E_8$ which is supported at the set of nodes complementary
to $\{n(x,y)\}$. If $z$ is a certain $2B$ involution which commutes with $\la x, y \ra$,
Glauberman and Norton
 give a lot of detail about $C(x, y, z)$.  In particular, they proposed that
$C(x, y, z)$ has a ``new'' relation to the extended $E_8$-diagram, namely that
$C(x, y, z)/O_2(C(x, y, z))$   looks roughly like ``half'' of the Weyl group
corresponding to the subdiagram  $n'(x,y)$.   This article shows that this new relation is not valid for the $6A$-case.  See our main theorems.   The relations for $3C$ and $5A$-cases are valid \cite{gl3cpath,gl5A}.

\subsection{About the proof}

The main idea of \cite{gl3cpath} is to transfer  a problem in group theory to a study of  certain subVOAs of the Moonshine VOA $V^\natural$ and some lattices of the Leech lattice.   Thus, our the articles on moonshine paths involve a mixture of techniques, finite group theory, internal analysis of lattices spanned by rootless $EE_8$-pairs and analysis of sublattices of the Leech lattice.
The bijection between $2A$ involutions of $\MM$ and conformal vectors of central charge $1/2$ (abbreviated as \cvcch) in the Moonshine VOA $V^\natural$  is foundational.   See the theory of Miyamoto involutions \cite{M4}.

The first observation is that the dihedral group $\la x,y\ra$ is uniquely determined
by the subVOA generated by the associated \cvcch $e'$ and $f'$ \cite{c1,LYY2}.
We noticed that the subdiagram $n'(x,y)$ defines an automorphism
$r=r(x,y)$ of exponential type in $Aut(V_{E_8})$ and one can construct a pair
of conformal vectors  $e$ and $f$ of central charge $1/2$ in
a lattice VOA $V_{\EE}$ by using $r$. We also explained in \cite{gl3cpath} that $r(x,y)$ is conjugate in $Aut(V_{E_8})$ to an automorphism $\hat{h}(x,y)$ in a torus normalizer in
$Aut(V_{E_8})\cong E_8(\CC)$. This approach led us to consider a pair of $\EE$-sublattices $M$ and $N$ in $E_8\perp E_8$. We showed that the pair $(M,N)$ can be isometrically embedded into the Leech lattice $\Lambda$ and  that the subVOA $\la e,
f\ra$ of $V_{\EE}$ generated by $e$ and $f$ can be embedded into the VOA
$V_\Lambda^+\subset V^\natural$.
Many properties of the dihedral group $\la \tau_e,
\tau_f\ra $ generated by the Miyamoto involutions can be studied by examining embeddings of the pair  $(M,N)$ in $\Lambda$. In particular, the centralizer $C(\tau_e, \tau_f, z)$ has a factor subgroup which looks like the common stabilizer of $M$ and $N$ in $O(\L)/\{\pm 1\}$  (see
Corollary \ref{Czef} and \ref{2E8DIH12}).

\subsection{Statements of main results.  }

\begin{mthm}[Proposition \ref{dih6b} and Theorem \ref{theorem3A}]
Let $x,y,z\in \MM$ such that $x, y\in 2A, xy\in 3A$, and $z\in
2B \cap C_{\MM}(x,y)$. Then the triple $(x,y,z)$ is unique up to conjugation by $\MM$. Moreover, there exist an $EE_8$ pair $(M,N)$ in $\L$ such that $M+N\cong \dih{6}{14}$ and the centralizer $C_{\MM}(x,y,z)$ has a homomorphism onto
$C_{O(\L)}(\la t_M, t_N\ra )/O_2(C_{O(\L)}(\la t_M, t_N\ra))$, which is isomorphic to ``half'' of the Weyl
group of type $A_2+ E_6$.
\end{mthm}

\begin{mthm}[Theorem \ref{6A1orbits} and \ref{6A2orbits}]
Consider triples $(x,y,z)$ so that  $x,y,z\in \MM$  $x, y\in 2A, xy\in 6A$, and $z\in
2B \cap C_{\MM}(x,y)$.
There are two orbits on the set of such triples $(x,y,z)$ under
conjugacy by  $\MM$.

\medskip

\textbf{Orbit $6A.1$:}  $(xy)^3\in O_2(C_\MM(z))$, and the triple $(x,y,z)$
is  conjugate to  $(\tau_{e_M}, \tau_{\varphi_{\a/2}(e_N)}, z)$, where $M$ and $N$ are $EE_8$-sublattices of $\L$, $M+N=\dih{6}{14}$ and $\a\in M\cap N(4)$.

\medskip

\textbf{Orbit $6A.2$:}
$(xy)^3\notin O_2(C_\MM(z))$, and the triple $(x,y,z)$
is  conjugate to  $(\tau_{e_M}, \tau_{e_N}, z)$, where $M$ and $N$ are $EE_8$-sublattices of $\L$ and $M+N=\dih{12}{16}$.
\end{mthm}

In all cases, the centralizer of $\la x, y, z\ra$ is  determined.

\begin{mthm}[Theorem \ref{theorem6A1}]
Suppose that the triple $(x, y, z)$ is in the orbit $6A.1$. Then $C_{\MM}(x,y,z)$ has a homomorphism onto
$$C_{Stab_{O(\L)}(\ZZ\a)} ( t_M, t_N)/ O_2(C_{Stab_{O(\L)}(\ZZ\a)} ( t_M, t_N))\cong PSU(4,2).2.$$
The kernel $K$ is a $2$-group of order $2^{11}$ and we have an exact sequence
\[
1\to \la z\ra \to K \to \{\varphi_\b\mid \b\in \L \text{  and  } \la \b, M+N\ra \in 2\ZZ\} \to 1.
\]
\end{mthm}

\begin{mthm}[Theorem \ref{theorem6A2}]
Suppose that the triple $(x, y, z)$ is in the orbit $6A.2$. Then the natural map $\th$ of  $C_{\MM}(x,y,z)$
to $C_{\MM}(z)/O_2(C_{\MM}(z))$ has image isomorphic to
$$C_{O(\L)}(t_M, t_N)/\la \pm 1\ra \cong (3\times 2.(Alt_4\times Alt_4).2).2.$$
The kernel $\tilde{K}$ of $\th$ is a group of order $2^{9}$ and the sequence
\[
1\to \la z\ra \to \tilde{K} \to \{\varphi_\b\mid \b\in \L \text{  and  } \la \b, M+N\ra \in 2\ZZ\} \to 1
\]
is exact.  (For a description of the group $C_{O(\L)}(t_M, t_N)$ as an index $2$
subgroup of $2.(Dih_6 \times O^+(4,3))$, see \refpp{theorem6A2} and
\refpp{contain5}.
\end{mthm}

We also show that $O(\L )$ has one orbit on ordered pairs $(M,N)$ of $EE_8$-sublattices so that $M+N$ has type $DIH_6(14)$ \refpp{unidih6}
and
one orbit on ordered pairs $(M,N)$ of $EE_8$-sublattices so that $M+N$ has type $DIH_{12}(16)$ \refpp{unidih12}.   There are analogous transitivity results in
\cite{gl3cpath}, \cite{gl5A}.

A striking feature of the Glauberman-Norton theory is that the stabilizer of a triple $(x,y,z)$ (modulo $O_2$) seemed to be roughly ``half'' the Weyl group
of the corresponding node of the extended $E_8$-diagram.    Our results so far confirm this for several nodes, \cite{gl3cpath,gl5A}  but this is not the case for the $6A$-node,
for either of the two orbits, 6A.1 or 6A.2.

The Weyl group associated to removal of the $6A$-node has shape $$Weyl(A_5) \times Weyl(A_2) \times Weyl(A_1)\cong Sym_5\times Sym_3 \times Sym_2.$$
The quotients $C_{\MM}(x,y.z)/O_2(C_{\MM}(x,y,z))$ are described in  Main Theorems 3 and 4.   Neither can be interpreted as half the above Weyl group.

Our moonshine path theories for the $3C$, $5A$, $3A$ and $6A$ cases have different degrees of confirmation of the Glauberman-Norton observations.   Aspects of the $3C$ path \cite{gl3cpath} are especially nice.

\bigskip

{\bf Acknowledgements.  }  First author acknowledges financial support from United States NSA grant H98230-10-1-0201 and hospitality from Academia Sinica in Taipei during visits in 2011 and 2012.

The second author thanks Taiwan National Science Council (NSC 100-2628-M-001-005-MY4)
and National Center for Theoretical Sciences of Taiwan for financial support

\begin{center}
{\bf \large Notation and Terminology}

\vspace{0.1cm}
\begin{tabular}{|c|c|c|}
   \hline
\bf{Notation}& \bf{Explanation} & \bf{Examples }  \cr
& & \bf{in text}\cr
   \hline\hline
$2A, 2B, 3A, \dots$ & conjugacy classes of the Monster: &  Introduction\cr
                  & the first number denotes the order & \cr
                  & of the elements and the second letter & \cr
                  & is arranged in descending order of & \cr
                  & the size of the centralizers &\cr \hline
$A_1,  \cdots , E_8$ &root lattice for root system &  \refpp{EE8}\cr
& $\Phi_{A_1}, \dots ,
\Phi_{E_8}$& \cr
   \hline
$AA_1, \cdots ,$ & lattice isometric to $\sqrt 2$ times &\cr
$ \EE$ &the lattice  $A_1, \cdots,  E_8$ & \refpp{EE8}\cr
\hline
$A\circ B$ & central product of  groups $A$ and $B$ & \refpp{ORp}     \cr \hline
\cvcch & conformal vector of central charge $1/2$ &  \refpp{dih6b}     \cr \hline
$\dg L$ & discriminant group of integral  & \refpp{Q}\cr
 &  lattice $L$: $\dg L = \dual L/L$& \refpp{Q'}\cr      \hline
$\dih{n}{k}$ & the sum of a $EE_8$ pair $(M,N)$ such  &  \refpp{Q} \cr
& that the corresponding SSD     & \cr
& involutions generate a dihedral of   & \cr
& order $n$ and $M+N$ has rank $k$ &\cr \hline
 $E$ & a particular $EE_8$-sublattice & \refpp{gmn} \cr \hline
 $\EE$-involution & SSD involution whose negated  &  \refpp{Q} \cr
& space is isometric to $\EE$ & \cr \hline
$e_M$ & a \cvcch associated to  & \refpp{def:phisubx} \cr
& an $EE_8$ sublattice $M$& \cr\hline
$G$ & complex isometry group of & \refpp{complexG}\cr
&the Coxeter-Todd lattice $K_{12}$& \cr \hline
$g_1$& an order 3 isometry in $O_3(O(K_{12}))$  & \refpp{Q} \cr \hline
$H$ &  the subgroup generated by complex   & \refpp{CR}\cr
& refections defined by norm $4$ vectors  of $K_{12}$ & \cr\hline
\end{tabular}
\end{center}

\begin{center}
\begin{tabular}{|c|c|c|}
   \hline
\bf{Notation}& \bf{Explanation} & \bf{Examples in text}  \cr
   \hline\hline
$J$ & a sublattice of $\L $ isometric to $K_{12}$ & \refpp{Q}\cr \hline
$\mathcal H$ & the hexacode & \refpp{CRandRSSD} \cr \hline
$K$ & a sublattice of $\L $ isometric to $K_{12}$ & \refpp{Q} \cr \hline
$K_{12}$ & the Coxeter-Todd lattice of rank 12&\refpp{coxetertodd}\cr \hline
$\L$ & the Leech lattice, rank 24 & \cr \hline
 $\MM$ &   Monster sporadic  group & Introduction \cr \hline
$\mu$& the natural surjection from & \refpp{zmuxi}\cr
&$C_\MM(z) \to Aut(V_\L^+)\cong 2^{24}.Co_1$& \cr\hline
$O(X)$ & the isometry group of  & \refpp{Q}, \cr
  &the quadratic space $X$&  \refpp{auto} \cr\hline
  $O_p(G)$ & largest normal $p$-subgroup of & Introduction \cr
  &  a finite group $G$ ($p$ is a prime) &  \cr\hline
  $O^{\vep}(2n,q)$ & orthogonal group of dimension& \refpp{CD1} \cr
  & $2n$ over $\FF_q$ of type
  $\vep = \pm$ & \cr \hline
$P_{N}$ & the orthogonal projection from  & \refpp{proj}\cr
& a lattice $L$ to $\QQ\otimes N^*$ & \refpp{M+N}\cr \hline
RSSD & relatively semi self dual & \refpp{rssd} \cr \hline
SSD &  semi self dual & \refpp{rssd} \cr \hline
  $t_M$ & SSD involution associated to $M$ & \refpp{Q} \cr \hline
  $\varphi_\a(u\otimes e^\b)$ & a \cvcch & \refpp{phisuba} \cr \hline
  $\tau_e$ & the Miyamoto involution defined  & \refpp{Miyamototau}\cr
  & by a \cvcch $e$&  \refpp{aa1inv}\cr \hline
$\xi$ & the natural surjection from  & \refpp{zmuxi} \cr
& $Aut(V_\L^+)\to O(\L)/\la \pm 1\ra$ & \cr \hline
$\o^\pm (\a)$ & a \cvcch associated to a norm 4 vector & \refpp{aa1type} \cr \hline
$z$ & an automorphism of $V^\natural$ such that & \refpp{zmuxi}\cr
& $z|_{V_\L^+}=1$ and $z|_{V_\L^{T,+}}=-1$& \cr \hline
 \end{tabular}
\end{center}


\section{Preliminary}

We review terminology about rational lattices and involutions.   For background, see \cite{gal}.

\begin{de}
Let $X$ be a subset of Euclidean space. Define $t_X$ to be the orthogonal
transformation which is $-1$ on $X$ and is $1$ on $X^\perp$.
\end{de}

\begin{de}\labttr{rssd}
A sublattice $M$ of an integral lattice $L$ is {\it RSSD (relatively semiselfdual)}
if and only if $2L\le M + ann_L(M)$. This implies that $t_M$
maps $L$ to $L$ and is equivalent to this property when $M$ is
a direct summand.

The property that $2\dual M \le M$ is called {\it SSD
(semiselfdual)}. It implies the RSSD property, but the RSSD
property is often more useful.
\end{de}

\begin{nota}\labtt{EE8}
We use $XX_n$ to denote the lattice which is isometric to $\sqrt{2}$ times the root lattice of type $X_n$. For example, $EE_8$ is the $\sqrt{2}$ times of the root lattice $E_8$.
\end{nota}

\subsection{Conformal vectors and the Monster}\labtt{sec:G2}

Next we shall recall some facts about conformal vectors of central charge $1/2$ (\cvcch) in the lattice type VOA $V_\L^+$ and the Moonshine VOA $V^\natural$.

\medskip

We use the standard notation for the lattice vertex operator algebra
\begin{equation}\labtt{VL}
V_L = M(1) \otimes \CC\{L\}
\end{equation}
associated with a positive definite even lattice $L$  \cite{FLM}. In particular,
${\mathfrak h}=\CC\otimes_{\ZZ} L$  is an abelian Lie algebra and we extend the
bilinear form to ${\mathfrak h}$ by $\CC$-linearity. Also, $\hat {\mathfrak
h}={\mathfrak h}\otimes \CC[t,t^{-1}]\oplus \CC k$ is the corresponding affine
algebra and $\CC k$ is the 1-dimensional center of $\hat{\mathfrak{h}}$. The
subspace $M(1)=\CC[\a_i(n)|1\leq i\leq d, n<0]$ for a basis $\{\a_1,
\dots,\a_d\}$ of $\mathfrak{h}$, where $\a(n)=\a\otimes t^n,$ is the unique
irreducible $\hat{\mathfrak h}$-module such that $\alpha(n)\cdot 1=0$ for all
$\alpha\in {\mathfrak h}$ and $n$ nonnegative, and
 $k$ acts as the scalar 1. Also, $\CC\{L\}=span
\{e^{\beta}\mid \beta\in L\}$ is the twisted group algebra of the additive group
$L$ such that $e^\b e^\a=(-1)^{\la \a, \b\ra} e^\a e^\b$ for any $\a, \b\in L$.
The vacuum vector $\mathbf{1}$  of $V_L$ is $1\otimes e^0$ and the Virasoro
element $\omega$ is $\frac{1}{2}\sum_{i=1}^d\beta_i(-1)^2\cdot \mathbf{1}$
where $\{\beta_1,..., \beta_d\}$ is an orthonormal basis of ${\mathfrak h}.$ For
the explicit definition of the corresponding vertex operators, we shall refer to
\cite{FLM} for details.

\begin{nota} \labtt{def:phisubx}
Let $M\cong \EE$. Define
\begin{equation}\labtt{eE}
 e_M=\frac{1}{16} \omega_M +\frac{1}{32}\sum_{\a\in M(4)} e^\a,
\end{equation}
where $\omega_M$ is the Virasoro element of $V_M$ and $M(4)=\{\a\in M|
\langle\a, \a\rangle = 4\}$. It is shown in \cite{dlmn} that $e_M$ is a simple
conformal vector of central charge $1/2$.

For $x\in M^*$, define a $\ZZ$-linear map
\[
\begin{split}
 \la x, \cdot\ra : M &\to \ \, \ZZ_2 \\
y & \mapsto \la x,y\ra  \mod 2.
\end{split}
\]
Clearly the map
\[
\begin{split}
\varphi: M^* & \longrightarrow \mathrm{Hom}_\ZZ(M, \ZZ_2)\\
         x &\longmapsto \la x, \cdot \ra
\end{split}
\]
is a group homomorphism and $Ker \varphi=2 M^* = M$.  For any $x\in
M^*=\frac{1}2 M$, $\la x, \cdot\ra $ induces an automorphism $\varphi_x$ of
$V_M$ given by
\labtt{phisubx}
\begin{equation}
\varphi_x(u\otimes e^\a) =(-1)^{\la x,\a\ra} u\otimes e^\a \quad
\text{ for } u\in M(1)\text{ and } \a \in M.
\end{equation}
Note that $ \varphi_x (e_M)$ is also a simple conformal vectors of central charge
$1/2$.
\end{nota}

\begin{rem}\labtt{Miyamototau}
Given a simple conformal vector $e$ of central charge $1/2$, one can define an
involutive automorphism $\tau_e\in Aut(V)$, called Miyamoto involution. If
$V=V^\natural$ is the Moonshine VOA, then $\tau_e$ defines a $2A$ involution
in the Monster \cite{M4}.
\end{rem}

\medskip

\begin{nota}\labtt{Ltheta}
Let $\L$ be the Leech lattice and $V_\L$  the lattice VOA associated with $\L$.
Let $\theta$ be a lift of the $-1$-isometry of $\L$ to $Aut(V_\L)$.
We use $V_{\Lambda}^{T}$ to denote the unique $\theta$-twisted module of $V_\L$
and  $V_{\Lambda}^+$,  $V_{\Lambda}^{T,+}$ to denote the fixed point subspaces of
$\theta$ in $V_\L$ and $V_\L^T$, respectively.
\cite{FLM}.
\end{nota}

\medskip

\begin{nota}\labtt{phisuba}
Let $\L$ be the Leech lattice. For each $\a\in L$, we define $\varphi_\a \in Aut(V_\L)$ by
\[
\varphi_\a(u\otimes e^\b) =(-1)^{\la \a,\b\ra} u\otimes e^\b \quad
\text{ for } u\in M(1)\text{ and } \b \in \L.
\]
Note that $\varphi_\a=\varphi_{\a'}$ if and only if $\a-\a'\in 2\L$. In addition, $\varphi_\a$ commutes with $\theta$ and thus $\varphi_\a$ also defines an automorphism on $V_\L^+$.
\end{nota}

\begin{nota}\labtt{zmuxi}
Let $\theta$, $V_{\Lambda}^+$, and $V_{\Lambda}^{T,+}$  be defined as in \refpp{Ltheta}.
Then the Moonshine VOA $V^\natural$  \cite{FLM} is constructed as a
$\ZZ_2$-orbifold of the Leech lattice VOA $V_\L$, i.e.,
\[
V^\natural=V_{\Lambda}^+\oplus V_{\Lambda}^{T,+}.
\]
Let $z$ be an involution of $V^\natural$ acting as $1$ and $-1$ on
$V_{\Lambda}^+$ and $V_{\Lambda}^{T,+}$ respectively.  Then  $z$ defines a
involution on $V^\natural$, which is in conjugacy class $2B$ \cite{FLM, M4}.
Recall from \cite{LS} that $$ C_{Aut(V^\natural)}(z)/\la z\ra \cong C_{Aut
(V_\L)}(\theta)/\la \theta\ra  \cong Aut (V_\L^+) $$ and the sequences
\[
1\longrightarrow \langle z\rangle\longrightarrow C_{Aut(V^\natural)}(z)
\overset{\mu}{\longrightarrow} Aut(
V_\Lambda^+)\to 1
\]
and
\[
1 \longrightarrow  \mathrm{Hom}(\Lambda, \ZZ_2) \longrightarrow
Aut(V_\Lambda^+)\ \overset{\xi}{\longrightarrow} \ O(\Lambda)/\langle\pm
1\rangle\longrightarrow  1.
\]
are  exact.
\end{nota}

The following results can also be found in \cite{LS}.
\begin{thm} [Theorem 5.18 of \cite{LS}]\labtt{aa1type}
Let $\a\in \L(4)$ and let
\[
e :=\omega^{+}(\a)= \frac{1}{16}\a(-1)^2\cdot 1 + \frac{1}{4} (e^\a +e^{-\a}).
\]
Then there is an exact sequence
\[
1 \to K  \to C_{Aut(V^\natural)} (z,\tau_e)/\la z\ra \to Stab_{O(\L)}(\ZZ \a)\cong Co_2\to 1,
\]
where $K \cong \{\varphi_\b \in Hom(\L, \ZZ_2) \mid \la \a, \b\ra \in 2\ZZ\}$.
\end{thm}

\begin{rem}\labtt{aa1inv}
Recall from \cite{LS} that for any $\a\in \L(4)$ and
\[
e =\omega^{\pm}(\a)= \frac{1}{16}\a(-1)^2\cdot 1 \pm \frac{1}{4} (e^\a +e^{-\a}),
\]
the Miyamoto involution $\tau_e$ acts on $V_\L^+$ as $\varphi_\a$ defined in \refpp{phisuba}.
\end{rem}

\begin{lem}\labtt{f10}
Let $M$ be a sublattice of $\Lambda$ isomorphic to $\EE$. Then the sequence
\[
1 \longrightarrow X \longrightarrow
Stab_{Aut(V_\Lambda^+)}(V_M^+) \overset{\xi}{\longrightarrow}
C_{O(\L)}(t_M)/\la \pm 1\ra \longrightarrow  1
\]
is  exact, where $X= \{\varphi_\a\in Hom(\Lambda, \ZZ_2)\mid \la \a , M\ra \in 2\ZZ\}$.
\end{lem}

\begin{thm}\labtt{thm:E8} Let $M$ be a sublattice of $\Lambda$ isomorphic to
$\EE$.  Let $e=\varphi_x(e_M)$ be defined as in Notation \refpp{def:phisubx}.
Then the centralizer $C_{\Aut V^\natural}(\tau_e,z)$   stabilizes the subVOA
$V_M^+$ and has the structure $2^{2+8+16}.\Omega^+(8,2)$. Moreover, the
map
$$\xi\circ \mu : C_{\Aut V^\natural}(\tau_{e},z) \to C_{O(\L)}(t_M)/\la \pm
1\ra$$ is surjective and $C_{\Aut V^\natural}(\tau_{e},z)$ acts on $V_M^+$ as
$\Omega^+(8,2)$, which is the quotient of the commutator subgroup of the
Weyl group of $E_8$ by its center.
\end{thm}

The following result can be found in \cite[Appendix F]{GL} (see also
\cite{gl3cpath,gl5A}).

\begin{thm}
Except for $\dih{4}{15}$, every $\EE$ pair $(M,N)$ in Table 1 of \cite{GL} can be embedded into $\L$.
\end{thm}

\medskip

\begin{de}\labtt{proj}
Let $L$ be an integral lattice and $N$ a sublattice. We denote the orthogonal projection of $L$ to $\QQ \otimes N^*$ by $P_{N}$.
\end{de}


\section{$3A$-triples}

In this section, we consider a triple of elements $x,y,z\in \MM$ such that $x, y\in 2A, xy\in 3A$ and $z\in
2B \cap C_{\MM}(x,y)$.  We shall show that there is only one orbit of such triples under the action of the Monster
and determine their centralizer in $\MM$.

\begin{lem}\labtt{xyztodih6a}
Let $x,y,z\in \MM$ be such that $x, y\in 2A, xy\in 3A$ and $z\in
2B \cap C_{\MM}(x,y)$. Then, there exists a pair  of $EE_8$-sublattices $(M,N)$ of $\L$ and $a\in M^*$, $b\in N^*$ such that
$M+N\cong \dih{6}{14}$ and the triple $(x,y,z)$ is conjugate to $(\tau_{\varphi_a(e_M)}, \tau_{\varphi_b(e_N)}, z)$ in $\MM$.
\end{lem}

\pf Up to conjugation, we may assume $z$ acts as $1$ on
$V_{\Lambda}^+$ and as $-1$ on $V_{\Lambda}^{T,+}$ and $(V^\natural)^z = V_\L^+$.

By the 1-1 correspondence between \cvcch in $V^\natural$ and 2A involutions  of $\MM$ \cite{Ho,M4}, we have
$x=\tau_e$ and $y=\tau_f$ for some \cvcch $e,f \in V^\natural$. Since $z$ centralizes $x$ and $y$, we have $\tau_{ze}=z\tau_ez^{-1} =\tau_e$ and $\tau_{zf}=z\tau_fz^{-1} =\tau_f$. Hence $e$ and $f$ are fixed by $z$ by the 1-1 correspondence.

Since $\tau_e\tau_f$ has order 3, both $e$ and $f$ must be of $EE_8$-type. That means there exists $EE_8$-sublattices $M$ and $N$ and $a\in M^*$ and $b\in N^*$ such that
$e=\varphi_a(e_M)$ and $f= \varphi_b(e_N)$ (see \refpp{def:phisubx}).

Recall that there are two types of \cvcch in $V_\L^+$. If $e=\omega^\pm(\a)$ is of $AA_1$-type, then $\tau_e= \varphi_\a$ on $V_\L^+$ and $\tau_e\in O_2(Aut(V_\L^+))= \{\varphi_\a\mid \a \in \L\} = \mathrm{Hom}(\Lambda, \ZZ_2)\cong 2^{24}$ (see \cite{LS},\refpp{phisuba} and \refpp{aa1inv}). Hence $\tau_f\tau_e\tau_f= \varphi_\b$ for some $\b\in \L$ and $(\tau_e\tau_f\tau_e\tau_f)^2= (\varphi_\a\varphi_\b)^2=1$. Therefore, $\tau_e\tau_f$ is of order $1,2$ or $4$.

Since $xy\in 3A$, $\xi\circ \mu(xy)= t_Mt_N$ has order $3$ and $\la e, f\ra =\frac{13}{2^{10}}$. Hence $M+N\cong \dih{6}{14}$ \cite{GL}.
\qed

\medskip

The following lemma can be obtained easily by direct calculation (see \cite{LYY2}, for example ).
\begin{lem}\labtt{inner}
Let $(M, N)$ be an $\EE$-pair such that $M+N\cong \dih{6}{14}$. Then
\[
  \la e_{M}, \varphi_b e_{N}\ra =
 \begin{cases}
\frac{13}{2^{10}} & \text { if } \la b, M\cap N \ra \in 2\ZZ, \\
\frac{5}{2^{10}} & \text { otherwise}.
\end{cases}
\]
\end{lem}

By Lemma \ref{inner}, we can refine the statement of Lemma \ref{xyztodih6a} as follows.

\begin{prop}\labtt{dih6b}
Let $x,y,z\in \MM$ be such that $x, y\in 2A, xy\in 3A$ and $z\in
2B \cap C_{\MM}(x,y)$. Then,  the triple $(x,y,z)$ is conjugate to $(\tau_{e_M}, \tau_{e_N}, z)$ in $\MM$,
where  $(M,N)$ is an $EE_8$-pair in $\L$ such that
$M+N\cong \dih{6}{14}$.
\end{prop}

\pf Since $(det(M),det(\L ))=1$ and $M$ is a direct summand of $\L$, there is $\b\in
\L$ such that $P_{M}(\b)\in a+2M$, where $P_{M}$ is the natural
projection from $\L$ to $M^*$. Then $\varphi_\b (\varphi_a(e_M))=e_M$. Hence, we may assume
$x=e_M$ and $y=\varphi_b(e_N)$. Since $\la x,y\ra = \la e_M, \varphi_b (e_N)= \frac{13}{2^{10}}$, we have $\la b, M\cap N\ra \in 2\ZZ$ by Lemma \refpp{inner} and $P_{(M\cap N)}(b)\in 2 (M\cap N)^*$.

Recall that $N^*=\frac{1}2 N$. Thus,  $b=\frac{1}2 \eta$ for some $\eta\in N$.
We may assume without loss that $\eta \in ann_{N}(M\cap N) \mod 2N$.
Let $h= t_Nt_M$. Then $h$ is of order $3$ and $h(N)=M$.
Define $\a = - h (\eta)\in M$. Then $\a = \frac{1}2( \eta + (h^2-h) \eta )$ and we have $P_{N}(\a) = b$. Thus,
\[
\varphi_\b(y)=e_N\quad \text{ and } \quad \varphi_\b(x)= \varphi_\b(e_M)=e_M
\]
as desired. \qed

\medskip

The next theorem is important to our study, which translates a problem
about $V^\natural$ to the study of the Leech lattice.

\begin{prop}\labtt{2E8}
Let $(M, N)$ be an $\EE$-pair in $\Lambda$ such that $M+N\cong \dih{6}{14}$. Then
\[
\xi\circ \mu:  C_{Aut(V^\natural)}(\tau_{e_{M}},\tau_{e_N},z)
\to C_{O(\L)}(t_M, t_N)/ \la \pm 1\ra
\]
is surjective.
\end{prop}

\pf Let $s\in C_{O(\L)}(t_M, t_N)/ \la \pm 1\ra$. Then by Lemma \refpp{f10}, there
is $g\in Aut(V^\natural)$ such that $g$ stabilizes both $V_M^+$ and $V_N^+$
and $\xi\circ \mu (g)=s$.

By the classification of \cvcch in $V_M^+$ \cite{O+102,LS},
$g(e_M)=\varphi_a(e_M)$ for some $a\in M^*$. By the same argument as in Proposition \ref{dih6b}, we
 may also assume $g(e_M)=e_M$.

Since $g$ also stabilizes  $V_N^+$, we have $g(e_N)= \varphi_b(e_N)$ for some
$b\in N^*$.  Moreover,
\[
\la e_M, \varphi_y e_N \ra = \la g(e_M), g(e_N)\ra =  \la e_M, e_N\ra = \frac{13}{2^{10}}.
\]
Thus, $\la b, M\cap N\ra \in 2\ZZ$ by Lemma \refpp{inner}. Then by the same argument as in Proposition \ref{dih6b},
there is an $\a \in M$ such that $P_{N}(\a) = b$. Hence
\[
\varphi_\a (g(e_N)) =e_N \quad \text{ and } \quad
\varphi_\a (g(e_M)) =\varphi_\a (e_M)=e_M.
\]
Therefore,  $g'= \tau_{\omega^+(\a)} g \in C_{Aut(V^\natural)}(\tau_{e_{M}},\tau_{e_N},z)$ and
$\xi\circ \mu (g')=s$. \qed

\begin{coro}\labtt{Czef}
Let $(M, N)$ be an $\EE$-pair in $\Lambda$ such that $M+N\cong \dih{6}{14}$.
The centralizer  $C_{Aut(V^\natural)}(\tau_{e_{M}}, \tau_{e_{N}},z)$ contains a
subquotient isomorphic to the common stabilizer of $M$ and $N$ in
$O(\L)/\{\pm 1\}$.
\end{coro}

\begin{proof}
Since $C_{O(\L)}(t_M, t_N)$ is the common stabilizer of $M$ and $N$ in $O(\L)$,
we have the conclusion by \refpp{2E8}.
\end{proof}

\medskip

\subsection{$\dih{6}{14}$}
Next we shall compute the group $C_{O(\L)}(t_M, t_N)$. First we recall some notations and facts about $\dih{6}{14}$ from \cite{GL}.

\begin{nota}\labtt{Q}
1.
Let $M$ and $N$ be $EE_8$-sublattices of the Leech lattice $\L$ such that
$Q:=M+N$ is isometric to $\dih{6}{14}$ as obtained in \cite{GL}.  We have
$F:=M\cap N\cong AA_2$.

Let $t_M$ and $t_N$ be the SSD involutions associated to $M$ and $N$, respectively.
Then the subgroup $D:=\la t_M, t_N\ra$ generated by $t_M$ and $t_N$ is a dihedral group $Dih_6$.

2. Let $J:=ann_{Q}(F)$. Then  $J\cong K_{12}$
is isometric to the Coxeter-Todd lattice and $Q$ contains a sublattice isometric to $F\perp J$.

3. $\dg{Q}= 2^2\times 3^5$.

4.
Set $g:= t_Mt_N$. Then $g$ has order $3$ and it acts on $\L$ with trace $6$.
Let $K:= Fix_{\L}(g)$ be the fixed point sublattice of $g$ in $\L$.
Then $ann_{\L}(K)=J$ (see (2)). Moreover, $K\cong J \cong  K_{12}$.

5.
Let $g_1$ be an isometry of order 3 in $O(\L)$ such that  $g_1$ acts fixed point free on $K$ but
acts trivially on $J$. In this case, $g$ and $g_1$ generate an elementary abelian group
of shape
$3^2$ and
$gg_1$ has trace $-12$ on $\L$.
\end{nota}

Next we recall some basic properties of the Coxeter-Todd lattice $K_{12}$.
 It is well-known (cf. \cite{cs,CS2}) that $K_{12}$ can also be viewed as a rank $6$ complex lattice over the ring of Eisenstein integers $\ZZ[\omega]$ as follows:

\begin{nota}\labtt{coxetertodd}
Let $\omega:=(-1+\sqrt{-3})/2$ be a primitive cubic root of unity and let $\mathcal{E}:= \ZZ[\omega]$ be the ring of Eisenstein integers. Then $\mathcal{E}/2\mathcal{E} \cong \mathbb{F}_4$. Let $\sigma: \mathcal{E} \to \mathcal{E}/2\mathcal{E}$ be the natural quotient map    and $\mathcal{H}$ the hexacode over $\mathcal{F}_4$. Then the Coxeter-Todd lattice can be defined as the sublattice
\[
K_{12}= \{ (x_1, \dots, x_6)\in \mathcal{E}^6\mid  (\sigma(x_1), \dots,\sigma(x_6))\in \mathcal{H}\}.
\]
The \textit{norm} of a vector $v$  in $K_{12}$ is defined by $\la v, v\ra $, where $\la \, , \,\ra $ is the standard Hermitian inner product on $\CC^6$.
\end{nota}

By direct calculation, it is easy to show that $K_{12}$ has $756$ vectors of norm $4$, $4032$ vectors of norm $6$ and $20412$ vectors of norm $8$ \cite{CS2}.

\begin{nota}\labtt{complexG}
We denote the complex isometry group of $K_{12}$ by $G$. In other word, $G$ is the set of all
complex linear automorphisms of $\CC^6$ that preserve the norm and stabilize $K_{12}$.
\end{nota}

\begin{rem}\labtt{lambda}
Let $\rho$ be a root of unity in $\ZZ[\omega]$, i.e., $\rho=\pm 1, \pm \omega$ or $\pm \omega^2$, and let $\lambda_\rho$ be the linear map defined by $ v\to \rho\cdot v$. Then $\lambda_\rho$ defines a complex isometry on $K_{12}$ and clearly, it is contained in the center of $G$.  
\end{rem}

The following result can be found in \cite{CS2}.

\begin{thm}\labtt{CR}
Let $H$ be the subgroup generated by complex refections defined by the minimal (norm 4) vectors  of $K_{12}$.  Note that there are $126(=756/6)$ such reflections. Then

1. $H$ acts transitively on the sets of vectors of norms $4$, $6$ and $8$, respectively.

2. $H=G$.

3. The order of $G$ is $2^{9}{\cdot} 3^7{\cdot}5{\cdot}7$.

4. The center of $G$ is given by $\{\lambda_\rho\mid \rho=\pm 1, \pm \omega, \pm \omega^2\}$ and  has order 6.
\end{thm}

\medskip

\begin{rem} \labtt{CRandRSSD} 1.  For each of the 756 minimal vectors $v\in K_{12}$, the sublattice $\ZZ[\omega]v \cong 2\ZZ[\omega]$ has exactly $6$ minimal vectors.
Moreover, $\ZZ[\omega]v \cong 2\ZZ[\omega]$ is isometric to $AA_2$ as an integral lattice.
Note also that $\ZZ[\omega] v \cong 2\ZZ[\omega]\cong AA_2$ is invariant under the action of $\lambda_\omega$. Therefore, we have exactly $126$  $\lambda_{\omega}$-invariant $AA_2$ sublattices in $K_{12}$ and a complex reflection on a minimal vector corresponds to a RSSD-involution associated to a $\lambda_{\omega}$-invariant $AA_2$-sublattice of $K_{12}$.

2. Let $\nu$ be the anti-automorphism defined by coordinatewise complex conjugation and let $\phi$ be the linear transformation on $\CC^6$ defined by the matrix
\[
\begin{pmatrix}
1&0&0&0&0&0\\
0&0&1&0&0&0\\
0&1&0&0&0&0\\
0&0&0&1&0&0\\
0&0&0&0&\omega&0\\
0&0&0&0&0&\bar{\omega}
\end{pmatrix}.
\]
Then the anti-automorphism $\nu\circ \phi$ defines an isometry
of the real lattice $K_{12}$. Note that $\nu\circ \phi$ preserves the hexacode $\mathcal{H}$ and is an involution (see Proposition 4.5 of \cite{G12}).  Adjoining this anti-automorphism to $G$ will give the real  isometry group of $K_{12}$ \cite[Section 4.9]{cs}.
\end{rem}

The next theorem follows from Theorem \ref{CR} and Remark \ref{CRandRSSD}.

\begin{lem} \labtt{H}
Let $g$ be an order 3 element in $O_3(O(K_{12}))$ and let $H$ be the
subgroup generated by RSSD involutions associated to $g$-invariant $AA_2$ sublattices of $K_{12}$.
Then $H$ is an index 2 subgroup of $O(K_{12})$ 
and its center has order 6.
Moreover, $H$ acts
transitively on the sets of vectors of norms $4$, $6$ and $8$, respectively.
\end{lem}

We also note that the discriminant group $\dg{K_{12}} = K_{12}^*/K_{12}\cong 3^6$. It forms a non-singular quadratic space of minus type with respect to the standard bilinear $\frac{1}{3}\ZZ/\ZZ(\cong \FF_3)$ form \cite{cs,CS2}. The isometry group $O(K_{12})$ acts on $\dg{K_{12}}$ as the  orthogonal group
$O^-(6,3)= \Omega^-(6,3).2^2=2.P\Omega^-(6,3).2^2$
and the kernel of the action is a subgroup of order $3$. Since $P\Omega^-(6,3) \cong PSU(4,3)$, we have the following result.

\begin{lem}[Section 4.9 of \cite{cs} and \cite{CS2}]\labtt{auto}
Let $K_{12}$ be the Coxeter-Todd lattice of rank $12$.  Then

1. $O(K_{12})$ has the order $2^{10}. 3^7.5.7$ and the shape  $ (6.PSU(4,3).2).2$.
2. The complex isometry group $G$ is an index 2 subgroup of $O(K_{12})$ and has the shape
$6.PSU(4,3).2$.
\end{lem}

\begin{lem}
Let $J\cong K_{12}$  and $g$ be defined as in Notation \ref{Q}. Let $H$ be the subgroup of $O(J)$ generated  by RSSD involutions associated to $g$-invariant $AA_2$ sublattices in $J$. Then $H$ acts transitively on the set of all $g$-invariant $AA_2$ sublattices in $J$.
\end{lem}

\begin{proof}
Let $A$ and $B$ be two $g$-invariant $AA_2$ sublattices. Let $\a$ be a norm 4 vector of $A$. Then there exists $h\in H$ such that $h\a\in B$ by Lemma \refpp{H}. Then $hA = span_\ZZ\{ h\a, h g\a\} =B$ since $B$ is $g$-invariant and $g\in Z(H)$.
\end{proof}

\begin{lem}\labtt{trans4}
The centralizer $C_{O(\L)}(g)$ is transitive on the set of all norm $4$ vectors in $J$.
\end{lem}

\begin{proof}
We shall use the notion of hexacode balance to
denote the codewords of the Golay code and the vectors in the Leech lattice \cite{G12,cs}.
Namely, we shall arrange the index set $\Omega=\{1,2, \dots, 24\}$ into a $4\times 6$ array such that the six columns form a sextet.

Since there is a unique conjugacy class of order $3$ element with trace $6$ on $\L$  \cite{atlas}, we may assume
\[
g= \begin{array}{|c c c c c c|}
\hline
& && &&  \\
\downarrow & \downarrow & \downarrow & \downarrow & \downarrow & \downarrow \\
\downarrow & \downarrow & \downarrow & \downarrow & \downarrow & \downarrow \\
\downarrow & \downarrow & \downarrow & \downarrow & \downarrow & \downarrow \\ \hline
\end{array}
\]
and
\[
J=\left \{ \left.
{\tiny
\begin{array}{|c|c|c|c|c|c|}
\hline
0&0&0&0&0&0\\ \hline
x_1& x_2& x_3&x_4&x_5&x_6 \\ \hline
y_1& y_2& y_3&y_4&y_5&y_6 \\ \hline
z_1& z_2& z_3&z_4&z_5&z_6\\ \hline
\end{array}}
\in \L \  \right | \
x_i+y_i+z_i = 0 \text{ for all } i=1,\dots, 6 \right\}.
\]
There are two types of norm $4$ elements in $J$:

\textbf{Type I}
\[
\pm \frac{1}{\sqrt{8}}\  {\tiny
\begin{array}{|r|r|r|r|r|r|}
\hline
0&0&0&0&0&0\\ \hline
4& 0&0&0&0&0 \\ \hline
-4& 0&0&0&0&0 \\ \hline
0&0&0&0&0&0 \\ \hline
\end{array}},
\]
where the $\pm 4$'s are supported on a tetrad  and on the second, third or fourth rows. There are $36$ vectors of this type;

\textbf{Type II}
\[
\pm \frac{1}{\sqrt{8}}\  {\tiny
\begin{array}{|r|r|r|r|r|r|}
\hline
0&0&0&0&0&0\\ \hline
2& 2&2&2&0&0 \\ \hline
-2&-2&-2&-2&0&0 \\ \hline
0&0&0&0&0&0 \\ \hline
\end{array}}.
\]
Each weight $4$ codeword of the Hexacode $\mathcal{H}$  will give $2^4$ vectors of this kind and there are
in total
$720 (=2^4\times 45)$ norm 4 vectors of this type.

Recall that the subgroup of $M_{24}$ that stabilizes the standard sextet and fixes the first row is given by automorphism group of the Hexacode $Aut(\mathcal{H})\cong 3.S_6$ \cite{cs,G12}.
Let $U$ be the subgroup generated by all $\varepsilon_\mathcal{O}$, where $\mathcal{O}$ is a union of any two tetrads of the standard sextet. Then $U\cong 2^5$  and we obtain a subgroup $U.Aut(\mathcal{H})\cong 2^5.(3.S_6)$, which commutes with $g$ and is transitive on each of these two types of norm $4$ vectors.

Now it remains to show there is an element in $C_{O(\L)}(g)$ which mixes these two types of vectors.

Let $\mathcal{T}=\{ T_1, \dots, T_6\} $ be a sextet given as below:
\[
{\small
\begin{array}{|r|r|r|r|r|r|}
\hline
*    & \circ &\bullet &\bullet&\bullet &\bullet \\ \hline
\circ & *    &\times & \times & \times &\times \\ \hline
\circ & *    & \diamond  &\diamond &\diamond  &\diamond\\ \hline
\circ & *    & & & &  \\ \hline
\end{array}}.
\]
We take $T_1$ to be the tetrad marked by $*$.

Let $\xi_{\mathcal{T}}$ be the linear map defined by
\[
\begin{split}
v_i & \to v_i -\frac{1}2 v_{T_1}\quad \text{ if }  i\in T_1,\\
v_i & \to \frac{1}2 v_{T} - v_i \quad \text{ if }  i\in T, T\in \mathcal{T}, T\neq T_1,
\end{split}
\]
where $\{\pm v_i\mid i\in \Omega=\{1, \dots,24\}\}$ is a standard frame of norm 8 vectors in  $\L$ and $v_S=\sum_{i\in S} v_i$ for any $S\subset \Omega$. Then $\xi_{\mathcal{T}}$ is an isometry of $\L$ (cf. \cite[p. 288]{cs} and \cite[p. 97]{G12}).

It is easy to see that $\xi_{\mathcal{T}}$ commutes with $g$ and that 
\[
\xi_T\left ( \ {\tiny
\begin{array}{|r|r|r|r|r|r|}
\hline
0&0&0&0&0&0\\ \hline
0& 0&0&0&0&4 \\ \hline
0& 0&0&0&0&-4 \\ \hline
0&0&0&0&0&0 \\ \hline
\end{array}}\,
\right)
=
{\tiny
\begin{array}{|r|r|r|r|r|r|}
\hline
0&0&0&0&0&0\\ \hline
0&0& 2& 2&2&-2 \\ \hline
0&0&-2&-2&-2&2 \\ \hline
0&0&0&0&0&0 \\ \hline
\end{array}}.
\]
Hence  $C_{O(\L)} (g)$ is transitive on the set of all norm 4 vectors in $J$.
\end{proof}

\begin{lem} [cf. \cite{cs,G12}]\labtt{trans}
The centralizer $C_{O(\L)}(g)$ is transitive on the set $\mathcal{A}=\{ A\subset J\mid A\cong AA_2\text{ and is $g$-invariant}\}$.
\end{lem}
\begin{proof}
Let $A, B\in \mathcal{A}$ and let $\a\in A$. Then by Lemma \refpp{trans4}, there exists $h\in C_{O(\L)}(g)$ such that $h \a \in B$. Since $h$ commutes with $g$ and $B$ is $g$-invariant, we have
$$ h A = span_{\ZZ}\{ h \a, hg\a\} = span_{\ZZ}\{ h \a, g h\a\} =B$$
 as desired.
\end{proof}

\begin{lem}\labtt{RSSDA2}
Let $A$ be a $g$-invariant $AA_2$-sublattice in $J$. There exists an $EE_8$-sublattice $E$ of $\L$ such that $E\cap J=A$. In this case, $E\cap K\cong EE_6$ \refpp{Q}.
\end{lem}

\begin{proof}
Clearly, there exists an $EE_8$-sublattice $X$ such that $X\cap J\cong AA_2$ (cf. Case 1:\,$\dih{6}{14}$ of Appendix A). By Lemma \refpp{trans}, $C_{O(\L)}(g)$ is transitive on $\mathcal{A}:=\{ A\subset J\mid A\cong AA_2\text{ and is $g$-invariant}\}$. Thus, there is an $h\in C_{O(\L)}(g)$ such that $h(X\cap J)=A$. Now take $E= hX$. Then $E\cap J = h X\cap J= h(X\cap J)=A$ since $hJ=J$.
\end{proof}

\begin{rem}\labtt{UniqE}
Let $E^1$ and $E^2$ be $EE_8$ sublattices of $\L$ such that $E^1\cap J=E^2\cap J\cong A$. Then $E^1+E^2 \cong \dih{6}{14}$ (cf. \cite{GL}). Moreover, $t_{E^1}t_{E^2}$ acts trivially on $J$ and hence $t_{E^1}t_{E^2}\in \la g_1\ra$. Therefore, $E^2=g_1^iE^1$ for some $i=1,2$.
\end{rem}

\begin{thm}\labtt{unidih6}
Let $M, N$ be $EE_8$-sublattices of $\L$ such that $M+N\cong \dih{6}{14}$. Then the pair $(M,N)$ is unique, up to the action of $O(\L)$.
\end{thm}

\begin{proof}
Let $(M,N)$ be such a pair. Then $A= M\cap N\cong AA_2$ and $g=t_Mt_N$ has order 3 and trace 6 on $\L$. Hence $K= Fix_{\L}(g) \cong K_{12}$ and $J= ann_{\L}(K)\cong K_{12}$.
Since such a $g$ is unique up to conjugacy, $K$ and $J$ are uniquely determined, up to the action $O(\L)$.

Now note that  $A\subset K$ since $A$ is the common $(-1)$-eigenlattice of $t_M$ and $t_N$. By Lemma \refpp{trans},  $(A, J)$ is unique, up to the action of $O(\L)$. Now by Lemma \refpp{RSSDA2} and Remark \refpp{UniqE}, the pair $(M,N)$ is unique up to the action of $O(\L)$.
\end{proof}

\begin{lem}[\cite{G12}]\labtt{Cg}
Let $g\in O(\L)$ be an element of order $3$ and trace $6$. Then
$C_{O(\L)}(g) $ has the shape $ (2\times 3^2).PSU(4,3).2$.
\end{lem}

\begin{proof}
First we note that $C_{O(\L)}(g) $ stabilizes both $K=Fix_\L(g)$ and $J=ann_\L(K)$ and hence
$C_{O(\L)}(g) $ acts on $J$ and induces a group homomorphism
$\varphi: C_{O(\L)}(g) \to C_{O(J)}(g) \cong H$.

Now let $A\in \mathcal{A}$. Then there is an $EE_8$-sublattice $E$ of $\L$ such that $E\cap J=A$.  In this case, the SSD involution $t_E$ acts as $t_A$ on $J$. Since the RSSD involutions $t_A, A\in \mathcal{A}$, generate $H$, we have $Im \varphi = H$.

\noindent \textbf{Claim:} $\ker \varphi = \la g_1 \ra$.

Proof. Let $\sigma\in O(\L)$ such that $\phi(\sigma)=id_J$, i.e., $\sigma$ fixes
$J$ pointwise. Thus $\sigma$ acts trivially on $J^*/J$.  Moreover, $\sigma$
acts on $K$ and must act trivially on $K^*/K$ since $\sigma$ preserves the gluing
map from $K\perp J$ to $\L$. By the discussion before Lemma \refpp{auto},
$\sigma \in \la g_1\ra$. Therefore, $C_{O(\L)}(g)/\la g_1\ra \cong H\cong
6.PSU(4,3).2$ and $C_{O(\L)}(g)$  has the shape  $(2\times 3^2).PSU(4,3).2$.
\end{proof}

\begin{rem}
We shall note that for any $EE_8$-sublattice $E$ of $\L$ such that $E\cap J\in \mathcal{A}$, the SSD involution $t_E$ inverts $g_1$ and so $t_E$ and $g_1$ generate a subgroup isomorphic to $Sym_3$.
\end{rem}

\begin{lem}\labtt{tranAA2perpAA2}
Let $g$ and $J$ be defined as in Notation \refpp{Q}. Then the centralizer $C_{O(\L)}(g)$ is transitive on the set
\[
\mathcal{B}=\{ A\perp B \subset J\mid A, B \text{ are $g$-invariant $AA_2$ sublattices of $J$}\}.
\]
\end{lem}

\begin{proof}
Let $H$ be the subgroup generated by RSSD involutions associated to $g$-invariant $AA_2$ sublattices as defined in Lemma \refpp{H}. Then $H$ is an index 2 subgroup of $O(J)$. By Lemma \refpp{Cg}, the image of $C_{O(\L)}(g)$ in $O(J)$ is $H\cong 6.PSU(4,3).2$.

Recall that $J/(g-1)J\cong J^*/J\cong 3^6$ is a non-singular quadratic space of $(-)$-type.
The group $O(J)$ acts on $J/(g-1)J$ as the full orthogonal group $O^-(6,3)\cong 2.PSU(4,3).2^2$ while $C_{O(\L)}(g)$ acts on $J/(g-1)J$ as $2.PSU(4,3).2$.

In $J/(g-1)J$, the image of a $g$-invariant $AA_2$ sublattice is a non-singular $1$-space and the image of $g$-invariant $AA_2\perp AA_2$ is a definite $2$-space. By Witt Theorem, $O^-(6,3)$ is transitive on definite 2-spaces. Therefore, $O(J)$ is transitive on $\mathcal{B}$.

Recall from Notation \refpp{coxetertodd} that
\[
(2\ZZ[\omega])^6\subset J \subset \ZZ[\omega]^6.
\]
Let $A\perp B$ be the sum of the first and fourth copies of $2\ZZ[\omega]\cong AA_2$. Then
the anti-automorphism $\nu\circ \phi$ defined in Remark \refpp{CRandRSSD} gives  an isometry of the real lattice $K_{12}$ and by definition, it stabilizes $A\perp B$. Hence the index
$[ Stab_{O(J)}(A\perp B) : Stab_{H}(A\perp B)]=2$ and $C_{O(\L)}(g)$ is transitive on $\mathcal{B}$.
\end{proof}

\begin{thm} \labtt{CDdih6}
Let $M$ and $N$ be defined as in Notation \refpp{Q} and let $D$ be the dihedral group generated by $t_M$ and $t_N$. Then
$C_{O(\L)}(D) \cong (3\times 2\times PSU(4,2)). 2$.
\end{thm}

\begin{proof}
Set $G = C_{O(\L)}(D)$. Since $g\in D$, $G$ stabilizes $K= Fix_\L(g)$ and $J=ann_{\L}(K)$. In addition, $G$ centralizes $t_M$ and hence it stabilizes the $(-1)$-eigenlattice of $t_M$, which is $M$. Therefore, $G$ acts on $F=M\cap K\cong AA_2$ and $M\cap J\cong EE_6$.

Let $\a\in M\cap J\cong EE_6$ be a norm 4 vector and $A= span_\ZZ\{\a, g\a\}$. Let $E$ be an $EE_8$ sublattice such that $E\cap J=A$. Then  $t_E$ acts as $t_{\la \a\ra}$ on $M\cap J$, which is a reflection at $\a$. Since $M\cap E\neq 0$ and $t_E$ commutes with $t_M$, we have
$M\cap E\cong AA_1^2$ or $DD_4$. Thus, $E\cap F= E\cap M \cap K\cong AA_1$ and $t_E$ acts as
a reflection at a root, also.  Moreover, the $-1$ map clearly acts on both $M\cap J$ and $F$. Thus, $G$ acts on $M\cap J$ as the full isometry group  $O(EE_6)=2\times Weyl(E_6) \cong 2\times PSU(4,2). 2$
and  acts as $O(AA_2)\cong 2. S_3$ on $F=M\cap K$. Hence $G= C_{O(\L)}(D)$ has the shape $(3\times 2\times PSU(4,2)). 2$.
\end{proof}

\begin{coro}\labtt{CDdih62}
Let $D$ be the dihedral group generated by $t_M$ and $t_N$. Then
$C_{O(\L)}(D)/O_2(C_{O(\L)}(D))\cong (3\times PSU(4,2)).2$, which is isomorphic to
``half'' of the Weyl group of type $A_2+ E_6$.
\end{coro}

The above results may be lifted to a statement about the
Monster as follows.
\begin{thm}\labtt{theorem3A}
Let $(x, y, z)$ be a triple of elements in the Monster such that $x, y\in 2A$, $xy\in 3A$, $z\in
2B \cap C_{\MM}(x,y)$. Then $(x,y,z)$ is unique up to the conjugation of $\MM$. Moreover, $C_{\MM}(x,y,z)$ has a homomorphism onto
$C_{O(\L)}(D)/O_2(C_{O(\L)}(D))$, which is isomorphic to ``half'' of the Weyl
group of type $A_2+ E_6$.
\end{thm}

\pf That the triple is unique up to conjugation follows from Proposition \ref{dih6b} and \ref{unidih6}.
The last statement follows from Corollary \ref{CDdih62}, Proposition \refpp{2E8} and Corollary \refpp{Czef}.\qed

\subsection{Isometry groups of $Q\cong \dih{6}{14}$ and $ann_\L(Q)$}

Let $Q\cong \dih{6}{14}$ and $R=ann_\L(Q)$. We shall determine the isometry group of $Q$ and $R$ in this subsection.

\begin{nota}\labtt{QJ}
Let $M$ and $N$ be $EE_8$ sublattices of $\L$  such that $Q=M+N\cong
\dih{6}{14}$ and let $R=ann_\L(Q)$.

Let $g=t_Mt_N\in O(\L)$. Set $K:=Fix_\L(g)$, $J:= ann_\L(Fix_\L(g))$ and
$F:=M\cap N$.  Then $K\cong J\cong K_{12}$, $F\cong AA_2$ and $F$ is
orthogonal to $J$. In addition, we denote the order $3$ fixed point free isometry
in $O_3(O(K))$ by $g_1$.
\end{nota}

We shall use the embedding of $Q$ and $R$ in $\L$ as discussed in Appendix A.
Then $R$ can be obtained by gluing $AA_2^5$ with the glue code
\[
\mathcal{C} =span_{\mathbb{F}_4} \{ (1,1,1,1,0),  (1,0, \omega, \omega^2, 1)\}.
\]
Note that  $\mathcal{C}$ has $15$ codewords of weight $4$ and $|R(4)|=
15\times 2^6 + 5\cdot 6= 270$. Moreover, $ann_{R}(A)\cong A_2\otimes D_4$ for
an $AA_2$-sublattice $A$ in $AA_2^5$.

\begin{lem}\labtt{ORp}
$O(A_2\otimes D_4)\cong O(A_2) \circ O(D_4)$, which has order $2^8\cdot 3^3$.
\end{lem}

\begin{proof}
It is clear that $O(A_2\otimes D_4)$ has a subgroup isomorphic to $O(A_2) \circ O(D_4)$. Since  the minimal vectors of $A_2\otimes D_4$ have the form
$\a\otimes \b$, where $\a$ and $\b$ are roots of $A_2$ and $D_4$, respectively \cite{GL},
we have $O(A_2\otimes D_4)\cong O(A_2) \circ O(D_4)$
\end{proof}

\begin{lem}\labtt{OR}
The order of $O(R)$ is  $2^8\cdot 3^5\cdot 5$.
\end{lem}

\begin{proof}

 Since $g_1$ is fixed point free on $R$,
$R$ has $45 (= 270/6)$ distinct $g_1$-invariant $AA_2$ sublattices.

Now let $A$ be an $AA_2$-sublattice in $AA_2^5$. Then the stabilizer $Stab_{O(R)} (A)$ of $A$
acts on the sublattice $R':=ann_{R}(A)\cong  A_2\otimes D_4$. Let $\varphi:
\dg{A} \to \dg{R'}$ be the gluing map. Then
\[
Stab_{O(R)} (A)= \{ (h, h') \in O(A)\times O(R')\mid \varphi\circ h = h'\circ \varphi\}.
\]
Let $p_{R'}: Stab_{O(R)} (A) \to O(R')\cong O(A_2)\circ O(D_4)$ be the restriction map. Then $\ker(p_{R'}) =\la t_A\ra\cong \ZZ_2$  and $Im(p_{R'})$ is an index $2$ subgroup of $O(R')$. Note that the field automorphism $\omega \to \omega^2$ defines an isometry on $R'$ but it does not lift to $O(R)$.

Therefore, $|Stab_{O(R)} (A)| = (|O(R')|\cdot 2)/2 = 2^8\cdot  3^3$ and hence $|O(R)|= |Stab_{O(R)} (A)|\cdot (45) = 2^8\cdot 3^5\cdot 5$.
\end{proof}

\begin{prop}
The isometry group $O(R)$ has the shape $ 6.PSU(4,2).2$.
\end{prop}

\begin{proof}
First, we note that $R=ann_{K_{12}}(A)$ for an $AA_2$ sublattice $A$ of
$K_{12}$. Hence, $Stab_{O(K_{12})}(A) = C_{O(K_{12})} (t_A)$ acts on $R$ with
the kernel $\la t_A\ra$.

Since $O(K_{12})\cong 6. PSU(4,3). 2^2$, by \cite[Page 52,53]{atlas},
$$C_{O(K_{12})} (t_A)\cong 6. PSU(4,2).2\times 2.$$
Hence $C_{O(K_{12})} (t_A)/ \la t_A\ra\cong 6.PSU(4,2).2$ acts faithfully on $R$.

Since $|6. PSU(4,2).2|= 2^8 \cdot 3^5\cdot 5$,  $O(R)$ has the shape $ 6. PSU(4,2).2$ by Lemma \refpp{OR}.
\end{proof}

\begin{thm}\labtt{OQ}
Let $Q\cong \dih{6}{14}$ and $J$ be defined as in \refpp{QJ}. Let $\varphi: O(Q)
\to O(J)$ be the restriction map. Then $Im(\varphi) \cong O(A_2)\circ O(E_6)$
and $Ker(\varphi)\cong Sym_3$. Therefore, $O(Q)$ has the shape
$Sym_3.(O(A_2)\circ O(E_6))$.
\end{thm}

\begin{proof}
Let $A$ and $J$ be defined as in \refpp{QJ}.  Then $O(Q)$ stabilizes both $A$ and $J$.
Hence, by Lemma \refpp{EE8inQ}, $O(Q)$  stabilizes the set $\{ M\cap J, gM\cap
J, g^{2}M\cap J\}$ and the sublattice $M\cap J +gM \cap J\cong A_2\otimes
E_6$. Recall that $M\cap J \cong EE_6$.

By Theorem \refpp{CDdih6}, $O(Q)$ acts as the full isometry group $O(E_6)$ on
the sublattices $M\cap J$, $gM\cap J $ and $g^2 M \cap J$, respectively. On the
other hand, $\la t_M, t_{gM}\ra\cong Dih_6$ acts  as permutations on the set
$\{ M\cap J, gM\cap J, g^{2}M\cap J\}$. Thus, we have $Im(\varphi) \cong
O(A_2\otimes E_6)=O(A_2)\circ O(E_6)$.

Clearly, $Ker(\varphi)$ can be viewed as a subgroup of $O(A)\cong \la t_A\ra.
Weyl(A_2)$.   Since $ [Q: A\perp J]=3$,  $A$ is not an RSSD in $Q$ and thus
$t_A$ is not an isometry of $Q$. Hence we have $Ker(\varphi)< Weyl(A_2)$.

Let $p_J: Q\to \dg{J}$ and $p_A: Q \to \dg{A}\cong 2^2\cdot 3$ be the natural
maps. Let $\xi: p_A(Q) \to p_J(Q)$ be the gluing map from $A\perp J$ to $Q$.
Since $O(Q)$ stabilizes $A\perp J$, we have
\[
O(Q)=\{ y\in O(A)\times O(J)\mid y
\text{ preserves the gluing map, } i.e., y \xi
=\xi y\}.
\]
Thus $y\in Ker(\varphi)$ if and only if $y\in O(A)\times O(J)$ acts trivially on $J$
and $y \xi =\xi y$.  Therefore, we have $Ker(\varphi)= \{ y\in O(A)\mid  \xi =\xi y
\}$.

By the discussion in Case 1 of Appendix A, we know that $p_A(Q)= 2A^*/A$ is a
subgroup of order $3$. Since $Weyl(A_2)$ fixes all cosets in $2AA_2^*/AA_2$,
we have  $Ker(\varphi) \cong Weyl(A_2) \cong Sym_3$.
\end{proof}

\begin{rem} \labtt{rssdA}
We note that $\ZZ\a$  is an RSSD sublattice
in $Q$ for any norm $4$ vector $\a\in A$.
The reason  is as follows:

Let $\a\in A$ with $(\a,\a)=4$.   We want to show that  $2Q\le ann_Q(\a) + \ZZ \a$. That is equivalent to $(\a, Q)\le 2\ZZ$.
First we notice that  the index $|Q:J+A|=3$ since $\dg J \cong 3^6$, $\dg A\cong 2^2\times 3$ and  $det(Q)=2^23^5$. Moreover, we have $(\a, J+A)\leq 2\ZZ$ since $A\cong AA_2$ is doubly even.   Also, there is an  integer $r$ so that $(\a,Q)=r\ZZ$.   Since $3Q\le J+A$, $3r\ZZ \le 2\ZZ$, whence $r$ is even.   We conclude that
$(\a, Q ) \le 2\ZZ$.
\end{rem}

\section{6A-triples}

In this section, we consider a triple $(x,y,z)$ in $\MM$ such that  $x,y\in 2A$, $xy\in 6A$ and $z\in 2B\cap C_\MM(x,y)$.
We shall study the orbits of such triples under the action of $\MM$.

\begin{nota}\labtt{sa}
Let $\SA=\{ (x,y,z)\mid x, y\in 2A, xy\in 6A, z\in
2B \cap C_{\MM}(x,y)\}$. Note that the  Monster $\MM$ acts on $\SA$ by conjugation.
\end{nota}

Take $(x,y,z)\in \SA$. Then $z\in 2B$ and we may again assume $z$ acts as $1$ on $V_\L^+$ and as $-1$ on $V_\L^{T,+}$ by conjugation. Moreover, $x=\tau_e$ and $y=\tau_f$ for some \cvcch $e$ and $f$ in $V_\L^+$, by the Miyamoto bijection \cite{Ho,M4} . By our assumption, $xy\in 6A$ and thus $(xy)^3\in 2A$. There are two cases:

1. $(xy)^3\in O_2(C_\MM(z))\cong 2^{1+24}$;

2.  $(xy)^3\notin O_2(C_\MM(z))$.

\medskip

\noindent \textbf{Case 1 $(6A.1)$:}  $(xy)^3\in O_2(C_\MM(z))\cong 2^{1+24}$.

In this case, $\xi\circ \mu(xy)$ has order $3$ in $O(\L)/\{\pm1\}$. Thus, by the same arguments as in \refpp{xyztodih6a} and \refpp{dih6b}, we may assume $e= e_M$ and $f=\varphi_b(e_N)$ for some $EE_8$-pair $(M,N)$ and $b\in N^*$ such that $M+N\cong \dih{6}{14}$. Since $xy\in 6A$, we have $\la e,f \ra = \frac{5}{2^{10}}$ and hence $\la b, M\cap N\ra \notin 2\ZZ$ by \refpp{inner}.  Then $b=\frac{1}2\a \mod 2N$ for some $\a\in M\cap N(4)$.

\begin{thm}\labtt{6A1orbits}
Let $(x,y,z)\in \SA$. Suppose $(xy)^3\in O_2(C_\MM(z))$. Then $(x,y,z)$ is  conjugate to $(\tau_{e_M}, \tau_{\varphi_{\a/2}(e_N)}, z)$ for some $EE_8$-pair $(M,N)$  such that $M+N=\dih{6}{14}$ and $\a\in M\cap N(4)$.
\end{thm}

\begin{prop}\labtt{DIH122}
Let $(M, N)$ be an $\EE$-pair in $\Lambda$ such that $M+N\cong \dih{6}{14}$.
Let $\a \in N(4)$ such that $\la \frac{1}2\a , M\cap N\ra \notin 2\ZZ$.
Then
\[
\xi\circ \mu:  C_{Aut(V^\natural)}(\tau_{e_{M}},\tau_{\varphi_{\frac{\a}2(e_N)}},z)
\to C_{Stab_{O(\L)}(\ZZ\a)} (t_M, t_N)/ \la \pm 1\ra
\]
is surjective.
\end{prop}

\begin{proof}
By Lemma \refpp{inner}, $\la e_M, \varphi_{\frac{\a}2} (e_N)\ra =\frac{5}{2^{10}}$ and $ \tau_{e_{M}},\tau_{\varphi_{\frac{\a}2 (e_N)}}$ generate a dihedral group of order $12$ in $Aut(V^\natural)$. In this case, $(\tau_{e_{M}}\tau_{\varphi_{\frac{\a}2 (e_N)}})^3 = \tau_{\omega^+(\a)} $ or $\tau_{\omega^-(\a)}$ \cite{LYY2}. Thus, the subgroup generated by $ \tau_{e_{M}},\tau_{\varphi_{\frac{\a}2 (e_N)}},z$ is the same as the group generated by $ \tau_{e_{M}},\tau_{e_N}, \tau_{\omega^+(\a)}$ and $z$. Recall that $\tau_{\omega^+(\a)} z =\tau_{\omega^-(\a)}$ \cite{LS}.
The result now follows by Lemma \refpp{aa1type} and Proposition \refpp{2E8}.
\end{proof}

\medskip

\noindent \textbf{Case 2 $(6A.2)$:}  $(xy)^3\notin O_2(C_\MM(z))$.

Then $\xi\circ \mu(xy)$ has order $6$ in $O(\L)/\{\pm1\}$. By the analysis in \cite{GL}, $e=\varphi_a(e_M)$ and $f=\varphi_b(e_N)$
 for some $EE_8$-pair $(M,N)$ in $\L$ such that $M+N\cong \dih{12}{16}$ and $a\in M^*$, $b\in N^*$

As in Proposition \ref{dih6b}, we may also assume $e=e_M$, up to conjugation.  Since $M\cap N=0$ and $\dih{12}{16}$ is a direct summand of $\L$, there is a $\b\in \L$ such that  $P_{N}(\b)=b \mod 2N $ and   $\la P_{M}(\b), M\ra \in 2\ZZ$ by \refpp{M+N}. Then, $\varphi_\b(f)=\varphi_\b (\varphi_b(e_N)) =e_N$ and $\varphi_\b(e)=\varphi_\b(e_M)= e_M$.

Thus, up to conjugation, $e= e_M$ and $f=e_N$ for some $EE_8$-pair $(M,N)$ such that $M+N\cong \dih{12}{16}$.

\begin{thm}\labtt{6A2orbits}
Let $(x,y,z)\in \SA$. Suppose $(xy)^3\notin O_2(C_\MM(z))$. Then $(x,y,z)$ is  conjugate to $(\tau_{e_M}, \tau_{e_N}, z)$ for some $EE_8$-pair $(M,N)$ such that $M+N\cong \dih{12}{16}$.
\end{thm}

\begin{prop}\labtt{2E8DIH12}
Let $(M, N)$ be an $\EE$-pair in $\Lambda$ such that $M+N\cong \dih{12}{16}$. Then
\[
\xi\circ \mu:  C_{Aut(V^\natural)}(\tau_{e_{M}},\tau_{e_N},z)
\to C_{O(\L)}(t_M, t_N)/ \la \pm 1\ra
\]
is surjective.
\end{prop}

\pf Let $s\in C_{O(\L)}(t_M, t_N)/ \la \pm 1\ra$. Then by Lemma \refpp{f10}, there
is $g\in Aut(V^\natural)$ such that $g$ stabilizes both $V_M^+$ and $V_N^+$
and $\xi\circ \mu (g)=s$.

As in Proposition \ref{2E8}, we may assume $g(e_M)=e_M$ and $g(e_N)= \varphi_b(e_N)$ for some
$b\in N^*$.  Since $M\cap N=0$ and $\dih{12}{16}$ is a direct summand of $\L$, there is a $\b\in \L$ such that  $P_{N}(\b)=b \mod 2N $ and   $\la P_{M}(\b), M\ra \in 2\ZZ$ by \refpp{M+N}. Then, $\varphi_\b (g(e_N)) =\varphi_\b(\varphi_b(e_N))=e_N$ and $\varphi_\b(g(e_M))= \varphi_\b(e_M)=e_M$

Therefore,  $g'= \tau_{\omega^+(\b)} g \in C_{Aut(V^\natural)}(\tau_{e_{M}},\tau_{e_N},z)$ and
$\xi\circ \mu (g')=s$. \qed

\subsection{Case $6A.1$ and $M+N\cong \dih{6}{14}$}

First we consider the case $6A.1$. In this case, $M+N\cong \dih{6}{14}$ and $e=e_M$ and $f=\varphi_{\frac{\alpha}2}(e_N)$ for some $\alpha \in (M\cap N) (4)$.

\begin{prop}\labtt{dih62}
Let $M$ and $N$ be defined as in Notation \refpp{Q} and let $\a\in M\cap N$ be a norm 4 vector. Then $C_{Stab_{O(\L)}(\ZZ\a)} ( t_M, t_N) \cong 2\times PSU(4,2).2$.
\end{prop}

\begin{proof}
First, we note that $ t_M, t_N$ stabilize $\ZZ\a$ since $\a\in M\cap N$. Moreover, the group
\[
C_{Stab_{O(\L)}(\ZZ\a)} ( t_M, t_N)= \{ h\in C_{O(\L)}(t_M, t_N)\mid h (\ZZ\a)=\ZZ\a \}.
\]
By Theorem \refpp{CDdih6}, $C_{O(\L)}(t_M, t_N)$ has the shape $2.(3\times PSU(4,2)).2$ which acts as $2.S_3$ on $M\cap N$ and acts as the isometry group of $ann_M(M\cap N) \cong EE_6$. Thus, the subgroup that fixes $\ZZ\a$ has the shape $2\times PSU(4,2).2$.
\end{proof}

By Lemma \refpp{f10} and Proposition \refpp{DIH122}, we have the corollary.

\begin{coro}\labtt{corollary3A}
Let $M$, $N$ and $\a$ be defined as in Proposition \refpp{dih62}. Let $z$ be the automorphism of $V^\natural$
such that $z|_{V_\L^+}=1$ and $z|_{V_\L^{T,+}}=-1$ as defined as in \refpp{zmuxi}. Then there is a homomorphism that maps $C_{Aut(V^\natural)}( \tau_{e_M},  \tau_{\varphi_{\frac{\a}2}(e_N)}, z)$ onto $$C_{Stab_{O(\L)}(\ZZ\a)} ( t_M, t_N)/ O_2(C_{Stab_{O(\L)}(\ZZ\a)} ( t_M, t_N))\cong PSU(4,2).2.$$
The kernel $K$ is a $2$-group of order $2^{11}$ and we have an exact sequence
\[
1\to \la z\ra \to K \to \{\varphi_\b\mid \b\in \L \text{  and  } \la \b, M+N\ra \in 2\ZZ\} \to 1.
\]
\end{coro}
\pf
To compute the kernel, note that the natural map $\L \rightarrow Hom(M+N,\ZZ_2)$ is onto since $M+N$ is a direct summand and $det(\L )=1$.   \eop

\medskip

By Theorem \ref{6A1orbits} and Corollary \ref{corollary3A}, we have our main theorem as follows.

\begin{thm}\labtt{theorem6A1}
Let $(x,y,z)$ be a triple of elements in the Monster such that $x, y$ in $2A$, $xy$ in $6A$, $z\in
2B \cap C_{\MM}(x,y)$ and $(xy)^3\in O_2(C_\MM(z))$. Then such triples form one orbit under the conjugation action of $\MM$ and
$C_{\MM}(x,y,z)$ has a homomorphism onto
$$C_{Stab_{O(\L)}(\ZZ\a)} ( t_M, t_N)/ O_2(C_{Stab_{O(\L)}(\ZZ\a)} ( t_M, t_N))\cong PSU(4,2).2.$$
The kernel $K$ is a $2$-group of order $2^{11}$ and we have an exact sequence
\[
1\to \la z\ra \to K \to \{\varphi_\b\mid \b\in \L \text{  and  } \la \b, M+N\ra \in 2\ZZ\} \to 1.
\]
\end{thm}

\subsection{Case $6A.2$ and $M+N\cong \dih{12}{16}$}

Next we consider the case $6A.2$. In this case, $M+N\cong \dih{12}{16}$, $e=e_M$ and $f=e_N$.

\begin{nota}[\cite{GL}]\labtt{Q'}
1. Let $M$ and $N$ be $EE_8$ sublattice of $\L$ such that $Q':=M+N$ is isometric to $\dih{12}{16}$ (Table 1 of \cite{GL}).
Then  $ann_M(N)\cong ann_N(M)\cong AA_2$.
 We also denote $R':=ann_{\Lambda}(Q')$.

2. Let $F':=ann_M(N)\perp ann_N(M)$ and  $J:=ann_{Q'}(F')$. Then  $J\cong K_{12}$
is isometric to the Coxeter-Todd lattice and $Q'$ contains a sublattice isometric to $F'\perp J$.

3. $\dg{Q'}\cong \dg{R'}\cong
6^4$.
\end{nota}

By explicit calculation in the Leech lattice (see Appendix A), one can show that $R'$ contains a sublattice isometric to $AA_2^{\perp 4}$ and $R'$ is isometric to
\[
span_\ZZ\{ AA_2^{\perp 4}, \frac{1}2 (\b_1, \b_1, \b_1, \b_1),
\frac{1}2(\b_2, \b_2, \b_2, \b_2)\},
\]
where $\b_1=\sqrt{2}\a_1$, $\b_2=\sqrt{2}\a_2$ and $\{\a_1,\a_2\}$ is a set of
fundamental roots for $A_2$.  In fact, $R'\cong A_2\otimes D_4$ (see \refpp{A2D4}).

\medskip

\begin{nota}[\cite{GL}] \labtt{dih12}
Let $t_M$ and $t_N$ be the SSD involutions associated to $M$ and $N$. Then the group
$\Delta:= \la  t_M, t_N\ra \cong Dih_{12}$.
Set $h:=t_Mt_N$, $g:= h^2$ and $u= h^3$. Then $h$ has order $6$, $g$ has order $3$ and $u$ has order $2$. The traces of $h$, $g$ and $u$ on $\L$ are $2$, $6$ and $8$, respectively. Note also that $\la u\ra =Z(\Delta)$ and $g=uh^{-1}$.
\end{nota}

\begin{prop}\labtt{unidih12}
Let $(M,N)$ be an $EE_8$-pair in $\L$ such that $M+N\cong \dih{12}{16}$. Then  the pair $(M,N)$ is unique up to  the action of $O(\L)$.
\end{prop}

\begin{proof}
Let $(M,N)$ be such a pair. Then $g=(t_Mt_N)^2$ has order 3 and trace 6 on $\L$. Let $K:= Fix_{\L}(g) \cong K_{12}$ and $J:= ann_{\L}(K)\cong K_{12}$.
Since such a $g$ is unique up to conjugacy, $K$ and $J$ are uniquely determined, up to the action $O(\L)$.

Let $g'$ be the fixed point free order $3$ element in $O_3(O(K))$. Then, $F_M:=M\cap K$ and $F_N:=N\cap K$ are both $g'$-invariant $AA_2$-sublattices in $K$ and  $F_M$ is orthogonal to $F_N$. By Lemma \refpp{tranAA2perpAA2},   $F_M\perp F_N\perp J$ is unique up to the action of $O(\L)$. Since $M+N$ is a direct summand of $\L$, we have $ann_\L(ann_\L(F_M\perp F_N\perp J)) =M+N\cong \dih{12}{16}$.

By Lemma \refpp{RSSDA2}, there exists $EE_8$ sublattices $E^1$ and $E^2$ such that $E^1\cap K= F_M$ and $E^2\cap K= F_N$. Then $E^1+E^2\cong \dih{12}{16}$ (cf. \cite{GL}) and thus $E^1+E^2= M+N$.  Therefore, $E^1=g^i M$ and $E^2=g^j N$ for some $i,j=0,1,2$ (see \refpp{EE8inQ'}) and $(E^1, E^2)$ is conjugate to $(M,N)$ by the action of the dihedral group $\la t_M, t_N\ra$. Hence $(M,N)$ is unique up to the action of $O(\L)$.
\end{proof}

Next we recall few facts about the lattice $\dih{12}{16}$ from \cite{GL}.

\medskip
\begin{lem}[\cite{GL}] \labtt{gmn}
Let $h$, $g$ and $u$ be defined as in Notation \refpp{dih12}.  Then

1. $N \cap g^{-1} M \cong DD_4$ and $N + g^{-1} M \cong \dih{4}{12}$.

2. $u= t_N t_{g^{-1}M} $ is an SSD involution associated to an $EE_8$ sublattice $E$
of $N + g^{-1}M$
\end{lem}

\begin{lem} [Proposition 6.44 of \cite{GL}] \labtt{Ng-1M}
Let $P=g^{-1}M\cap J$ and let $P^-(t_N)$ be the $(-1)$-eigenspace of $t_N$ in
$P$. Then $P^-(t_N) = N\cap g^{-1}M \cong  DD_4$ and $N\cap g^{-1}M < J$.
\end{lem}

\begin{lem}\labtt{CD1}
Let $\Delta=\la t_M, t_N\ra$ be defined as in Notation \refpp{dih12}.
The centralizer $C_{O(\L)}(\Delta)$ has the shape
$$ (3\times2 \times 2.(Alt_4\times Alt_4).2). 2,$$
which is an index 2 subgroup of $2.(Dih_6\times O(D_4))$.
\end{lem}

\begin{proof}
Set $G_1:= C_{O(\L)}(\Delta)$.

Let $L=M+gM$. Then $L\cong \dih{6}{14}$ and the group $\la t_M, t_{gM} \ra \cong Dih_{6}$ is a subgroup of $\Delta$. By the analysis of $\dih{6}{14}$, the centralizer
\begin{equation}\labtt{C1}
C_1:= C_{O(\L)}(\la t_M, t_{gM} \ra )\cong (3\times 2\times PSU(4,2)).2
\end{equation}
and it stabilizes $g^i M\cap J$ for each $i=0,1,2$  and acts as $O(EE_6)(\cong 2\times PSU(4,2).2)$ on each $g^i M\cap J$, where $J= ann_{\L} (K)$ and $K=Fix_{\L}(g)$.

Note that $G_1$ also centralizes $t_N$. Thus  $G_1$ stabilizes $N$ and $N\cap
g^{-1}M \cong DD_4$.  By Lemma \ref{Ng-1M},  $N\cap g^{-1}M< J\cap
g^{-1}M\cong EE_6$.  Therefore, $G_1$ acts as the stabilizer of $N\cap g^{-1}M
\cong DD_4$ on $g^{-1}M\cap J$, which is isomorphic to $2 . O(DD_4)\cong
2.O(D_4)$. Note that $O(D_4)\cong Weyl(F_4)\cong O^+(4,3)$ and has the
shape $2.(Alt_4\times Alt_4).2^2$ (see \refpp{WF4} and Appendix B of
\cite{gl5A}).

Since $u$ commutes with $D=\la t_M, t_{gM}\ra$, we have
\[
G_1 = C_{O(\L)}(D_1) = C_{O(\L)}( \la u, t_M, t_{gM} \ra)= C_{C_1} (u).
\]

Next we study the action of $u$ on $(g^{-1} M)\cap J$. First we note that  $u=
t_E$ acts as $-1$ on $g^{-1}M \cap K\cong AA_2$ and $E\cap g^{-1}M
\cong DD_4$. Thus $E\cap g^{-1}M \cap J \cong 2A_2$. Notice that  $E\cap g^{-1}M \cap J $ is the $(-1) $-eigenlattice of $u$ in $g^{-1}M\cap J$  and $N\cap g^{-1}M$ is the fixed point sublattice of $u$ in $g^{-1}M\cap J$. Hence $u$ acts as  $-t_{N\cap g^{-1}M}$ on $(g^{-1} M)\cap J$. Thus we have
$$C_{C_1} (u)/ \la -1, g_1 \ra = C_{C_1}(t_{N\cap g^{-1}M})/\la -1, g_1\ra $$
since $-1$ is in the center of $C_1$.   Recall that $g_1$ is an order 3 isometry of $\L$ as defined in (5) of
\refpp{Q}, which acts fixed point free on $K$ and trivially on $J$. Note also that
 $C_1/\la -1, g_1\ra \cong PSU(4,2).2 \cong Weyl(E_6)$
by \refpp{C1}.
Since $t_{N\cap g^{-1}M}$ has trace $-2$ on $g^{-1}M\cap J$, by the character table of $PSU(4,2)$ \cite[Page 26]{atlas}, we know that   $C_{C_1} (u)/ \la -1,
g_1\ra$ has the order $2\times 576$ and has the shape $2.(Alt_4\times Alt_4).2^2$.

Thus, $C_{O(\L)}(\Delta)$ has the order $12 \cdot 576= 2^9\cdot
3^3$ and has the shape $$(3\times2 \times 2.(Alt_4\times Alt_4).2). 2$$
as desired.
\end{proof}

By Lemma \refpp{f10}, Theorem \ref{6A2orbits} and Proposition \ref{2E8DIH12}, we deduce the main theorem of this section.

\begin{thm}\labtt{theorem6A2}
Let $(x,y,z)$ be a triple of elements in the Monster such that $x, y$ in $2A$, $xy$ in $6A$, $z\in
2B \cap C_{\MM}(x,y)$ and $(xy)^3\notin O_2(C_\MM(z))$. Then such triples form one orbit under the conjugation action of $\MM$ and $C_{\MM}(x,y,z)$
has a homomorphism onto
$$C_{O(\L)}(\Delta)/\la \pm 1\ra \cong (3\times 2.(Alt_4\times Alt_4).2).2.$$
The kernel $\tilde{K}$ is a group of order $2^{9}$ and the sequence
\[
1\to \la z\ra \to \tilde{K} \to \{\varphi_\b\mid \b\in \L \text{  and  } \la \b, M+N\ra \in 2\ZZ\} \to 1
\]
is exact.
\end{thm}

\pf The kernel can be computed by the same argument as in Corollary \refpp{corollary3A}. \qed

\subsection{Isometry groups of $Q'\cong \dih{12}{16}$ and $ann_\L(Q')$}

\begin{nota}\labtt{Q'R'}
Let $Q'= \dih{12}{16}$ and $R'=ann_{\L}(Q')$.  Let
$M$ and $N$ be $EE_8$ sublattices of $Q'$ such that $Q'=M+N$.
As in Notation \refpp{dih12}, we set $h=t_Mt_N$, $g= h^2$ and $u= h^3$.
Then $\la u\ra$ is the center of $\la t_M, t_N\ra$.
Let $E$ be the $(-1)$-eigenlattice of $u$ in
$Q'$.
Then $E$ is also an $EE_8$-sublattice of $Q'$ and $E\cap M \cong E\cap N\cong DD_4$ (see
\cite{GL}).

Set $F_M:=ann_M(N)$ and $F_N:= ann_N(M)$. Then by \cite{GL}, we have
$F_M\cong F_N\cong AA_2$ and $J:= ann_{Q'}( F_M\perp F_N)\cong K_{12}$. In
addition, $E\cap J\cong AA_2\perp AA_2$.
\end{nota}

\begin{lem}\labtt{A2D4OR'}
The isometry group  $O(R')\cong O(A_2) \circ O(D_4)$ and has the order
$2^8\cdot 3^3$.
\end{lem}

\begin{proof}
Since $R'\cong A_2\otimes D_4$ \refpp{A2D4}, we have $O(R')\cong
O(A_2\otimes D_4)\cong O(A_2)\circ O(D_4)$ by \refpp{ORp}.
\end{proof}

\begin{prop}
Let $M,N$ be defined as in \refpp{Q'R'}. Then the image of $\la t_M, t_N\ra$ in $O(Q')$ is normal in $O(Q')$.
\end{prop}

\begin{proof}
It follows from the classification of $EE_8$-sublattices in $Q'$ ( Lemma
\refpp{EE8inQ'}).
\end{proof}

\begin{prop}
Let $y\in O(Q')$. Then $y$ normalizes the subgroup $\la g\ra$ and thus it stabilizes the sublattice $J=ann_{Q'}(Fix_{Q'}(g))$.
\end{prop}

\begin{proof}
Since $y$ normalizes $\la t_M, t_N\ra$ and  $\la g\ra$ is the unique subgroup of order $3$ in $\la t_M, t_N\ra\cong Dih_{12}$,
$y$ also normalizes $\la g\ra$.
\end{proof}

\begin{rem}\labtt{R''}
By the discussion in Case 2 of
Appendix A, it is easy to see that $R'':=ann_{Q'}(E)= ann_{J}(E\cap J)\cong A_2\otimes
D_4\cong R'$.
\end{rem}

\begin{thm}
Let $Q'\cong \dih{12}{16}$ and $E$ be  defined as in
Notation \refpp{Q'R'}. Let $R''= ann_{Q'}(E)$ and let $\varphi': O(Q') \to O(R'')$ be the restriction map.
Then $Im(\varphi')=O(R'') \cong O(A_2)\circ O(D_4)$ and $Ker(\varphi')\cong
2.O^+(4,2)$. Therefore, $O(Q')$ has the shape  $2.O^+(4,2).(O(A_2)\circ
O(D_4))$.
\end{thm}

\begin{proof}
First we note that  $O(Q')$ stabilizes the $EE_8$ sublattice $E$ and hence it also
stabilizes $R''= ann_{Q'}(E) \cong A_2\otimes D_4$. Let $p_E: Q' \to \dg {E}$ and
$p_{R''}: Q' \to \dg {R''}$ be the natural maps and let $\xi': p_E(Q') \to
p_{R''}(Q')$ be the gluing map from $E\perp R''$ to $Q'$.  Then
\[
O(Q')=\{ y\in O(E)\times O(R'')\mid  y \xi'
=\xi ' y\}.
\]

By Lemma \refpp{CD1},  we know that $$R'' \cap N=R''\cap g^{-1}M = N\cap
g^{-1}M \cong DD_4$$ and $O(Q')$ acts as the full isometry group $O(D_4)$ on
$R''\cap N$. Similarly, we also have  $R''\cap g^i N (= R''\cap g^{i-1} M) \cong DD_4$ and $O(Q')$ acts
as $O(D_4)$ on each of $R''\cap g^i N$ for $i=0,1,2$. Moreover, the dihedral
group $\la t_N, t_M\ra $ acts on $R''$ as $Dih_6$ with the kernel $\la t_E\ra$.
More precisely, $\la t_N, t_M\ra $ acts as permutations on the set
$\{ R''\cap M , R''\cap gM, R''\cap g^2M\} = \{ R''\cap N, R''\cap gN , R''\cap g^2N\}$.
Hence, $O(Q')$ acts on $R''$ as the full isometry group $O(R'')\cong O(A_2)\circ
O(D_4)$ and we have
\[
Im(\varphi')= O(R'') \cong O(A_2)\circ O(D_4).
\]
In this case, the kernel of $\varphi'$ is given by
\[
Ker(\varphi') =\{ y\in O(E)\mid  y \text{ fixes }  p_E(Q') \text{ pointwise }\}.
\]
By our discussion in Case 2 of Appendix A, $[Q': E\perp R'']= 2^4 $ and $p_E(Q')$
forms a $4$-dimensional non-degenerate quadratic spaces of $(+)$-type; in fact,
it is isometric to the $2$-part of $\dg{AA_2\perp AA_2}$ and is a direct sum of
two non-singular 2-dimensional quadratic spaces of $(-)$-type. Recall that
$O(E)\cong 2. O^+(8,2)$ and hence the subgroup $Ker(\varphi')$ that fixes $p'(Q')$  pointwise is, by Witt's theorem,
isomorphic to $2.O^+(4,2)$.
This group $Ker(\varphi')$ is actually isomorphic to a direct product $2\times O^+(4,2)$ (for if $Y$ denotes the normal subgroup of $Ker(\varphi')$ generated by reflections on $R$, $Y\cong Dih_6 \times Dih_6$.   Thus $Ker(\varphi')/Y$ has order 4.  Now use \refpp{liftinvollatticemod2}).
\end{proof}

On the lattice $ann_{Q'}(J)\cong AA_2\perp AA_2$, $t_M$ negates one of the
$AA_2$ summands and fixes the other, while $t_N$ behaves analogously,
negating the summand which $t_M$ fixes. It is clear from the analysis of
$Ker(\varphi ')$ that there exists an element of $Ker(\varphi ')$ which
interchanges the two $AA_2$ summands. Therefore, since
 $\la t_M, t_N \ra$ is normal in $O(Q')$,
$O(Q')$ induces by conjugation the full automorphism group of $\la t_M, t_N \ra$ ($Aut(Dih_{12})\cong Dih_{12}$).

\begin{coro} \labtt{shapeO(Q')}
$O(Q')$ leaves invariant the sublattice $E\perp R''$ and the restriction maps to $E, R'$ give an embedding of
$O(Q')$ in   $(2.[Dih_6 \times O^+(4,2)])  \times (Dih_6 \times O^+(4,3))$.
\end{coro}

\begin{rem}\labtt{shapeO(Q')}
We have $O(D_4)\cong Weyl(F_4)\cong O^+(4,3)$.

The group $O(Q')$ acts on $\dg{Q'}\cong 6^4\cong
2^4\times  3^4$ as the full isometry group $O^+(4,2)\times O^+(4,3)$ and the
kernel of the action is given by $\la t_M, t_N\ra \cong
Dih_{12}$.
\end{rem}

\appendix
\section{Embeddings of $\dih{6}{14}$ and $\dih{12}{16}$ into the Leech lattice}

Let $L\cong AA_2^{12}$ be the orthogonal sum of $12$ copies of $AA_2$ and let $\mathcal{H}$ be the hexacode over $\FF_4=\{0,1,\omega, \omega^2\}$  with the generating matrix
\[
\begin{pmatrix}
1&1&1&1&0&0\\
0&0&1&1&1&1\\
1&\omega &0 &\omega^2 & 0&\omega^2
\end{pmatrix}
\]

Let $\mathcal{D}$ be the ternary code generated by
\[
\left(
\begin{array}{cccccccccccc}
1&1 &0&0 &0&0 &1&0 &0&0 &0&0 \cr
1&2 &0&0 &0&0 &0&1 &0&0 &0&0 \cr
0&0 &1&1 &0&0 &0&0 &1&0 &0&0 \cr
0&0 &1&2 &0&0 &0&0 &0&1 &0&0 \cr
0&0 &0&0 &1&1 &0&0 &0&0 &1&0  \cr
0&0 &0&0 &1&2 &0&0 &0&0 &0&1
\end{array}
\right).
\]
Then $\mathcal{D}$ is isomorphic to a sum of 3 ternary tetra-codes.

We can construct the Leech lattice by using $L$ and $\mathcal{D}$ and $\mathcal{H}
\perp \mathcal{H}$ as glue codes~\cite{KLY}.

\textbf{Case 1:} $\dih{6}{14}$.

Let $X$ be the sum of the first 5 copies of $AA_2$. Then by the above construction,
\[
Q:=\dih{6}{14} \cong ann_\Lambda( X).
\]
Moreover,
\[
R:=ann_\Lambda(Q)= ann_\Lambda(S),
\]
where $S$ is the sum of the last 7 copies of $AA_2$.

Note that $[R:AA_2^5]=2^4$. In fact, one can glue $AA_2^5$ to $R$ by using the
glue code generated by
\[
\begin{pmatrix}
1&1&1&1&0\\
1&\omega &\omega^2 & 0&\omega^2
\end{pmatrix}
\]

\[
\begin{array}{ccccc}
&& \Lambda & &\\
& \swarrow & &\searrow&\\
&EE_8^3 &  &  J' \perp J&\\
\mathcal{D}& \searrow& &\swarrow& \mathcal{C}' \perp \mathcal{C}\\
&& AA_2^{12}&&
\end{array}
\]

\textbf{Case 2:} $\dih{12}{16}$.

Let $Q':=\dih{12}{16}$ and $R':=ann_{\Lambda}(Q')$. Then $\dg{Q'}\cong \dg{R'}\cong
6^4$.

Let $X'$ be the sum of the 1st, 2nd,
3rd and 4th copies of $AA_2$. Then
\[
 \dih{12}{16}\cong ann_{\Lambda}( X').
\]

By explicit calculation in the Leech lattice, one can show that $R'$ contains a
sublattice isometric to $AA_2^{\perp 4}$ and $R'$ is isometric to
\[
span_\ZZ\{ AA_2^{\perp 4}, \frac{1}2 (\b_1, \b_1, \b_1, \b_1),
\frac{1}2(\b_2, \b_2, \b_2, \b_2)\},
\]
where $\b_1=\sqrt{2}\a_1$, $\b_2=\sqrt{2}\a_2$ and $\{\a_1,\a_2\}$ is a set of
fundamental roots for $A_2$.  Note that $\dg{AA_2}=2\cdot 6$.

By the analysis in \cite{GL}, $Q'$ contains a sublattice
\[
AA_2\perp AA_2 \perp J,
\]
where $J$ is isometric to the Coxeter-Todd lattice.

\begin{rem}\labtt{A2D4}
The lattice $R'$ is isometric to $A_2\otimes D_4$.
\end{rem}

\begin{proof}
Recall that  $R'$ is isometric to
\[
span_\ZZ\{ AA_2^{\perp 4}, \frac{1}2 (\b_1, \b_1, \b_1, \b_1),
\frac{1}2(\b_2, \b_2, \b_2, \b_2)\},
\]
where $\b_1=\sqrt{2}\a_1$, $\b_2=\sqrt{2}\a_2$ and $\{\a_1,\a_2\}$ is a set of
fundamental roots for $A_2$.  Then the sublattices
\[
\begin{split}
D^1= &span\{ (\b_1,0, 0,0), (0, \b_1, 0,0), (0,0,\b_1,0), \frac{1}2 (\b_1, \b_1, \b_1, \b_1) \} \quad \text{ and } \\
D^2= &span\{ (\b_2,0, 0,0), (0, \b_2, 0,0), (0,0,\b_2,0), \frac{1}2 (\b_2, \b_2, \b_2, \b_2)\}
\end{split}
\]
are both isometric to $DD_4$ and $D^2 =g(D^1)$, where $g=(\sigma,  \sigma,
\sigma,  \sigma)$ and $\sigma\in O(AA_2)$ is given by  $\b_1\to \b_2 \to
-(\b_1+\b_2) \to \b_1$.  Then $D^1+D^2\cong A_2\otimes D_4$ by Lemma 3.1
\cite{GL}. Hence $R'\cong  A_2\otimes D_4$ since they have the same
determinant.
\end{proof}

In this appendix, we shall determine all $EE_8$-sublattices in $Q'\cong\dih{12}{16}$ and $Q\cong \dih{6}{14}$.

\begin{lem}\labtt{CThasnoEE8}
Let $J$ be isometric to the rank 12 Coxeter-Todd lattice $K_{12}$. Then $J$ contains no sublattices isometric to $EE_8$.
\end{lem}

\pf Suppose $Y\cong EE_8$ is a sublattice of $J$. Then by Lemma A.3 of
\cite{GL}, the $3$-rank of  $\dg{Y}$ is at least $6+8-12=2$. It is a
contradiction since $\dg{Y}\cong 2^8$. Note that the $3$-rank of $\dg{J}=6$,
$rank(J)=12$ and $rank(E)=8$. \qed

\begin{lem}\labtt{A}
Let $A$ be a rank $2$  even lattice with $\dg{A}\cong 2^2\cdot 3$ and minimal
norm at least 4. Then $A\cong AA_2$.
\end{lem}

\begin{proof}
Let $\binom{ 2a\ \ c}{\ c \ \ 2b}$ be the Gram matrix of $A$. Then $4 ab
-c^2=12$ or $ab-(c/2)^2=3$. Hence, $c/2$ is  an integer. Therefore, $A'=
\frac{1}{\sqrt{2}} A$ is integral and $\det(A')=3$.

Let $H(n,d)= (\frac{4}3)^{(n-1)/2} d^{1/n}$ be the Hermite function (cf. \cite{gal,Kn}). Then
$H(2,3)=2.000$. Since $A$ has minimal norm at least 4, the minimal norm of
$A'$ is at least 2. Thus  $A'$ has a norm 2 vector $u$. Then $ann_{A'}(u)$ has
determinant $\frac{1}2\det(A')=3/2$ or $2\det(A')=6$. Since $A'$ is integral,
$ann_{A'}(u)$ has determinant $6$. Let $v$ be a basis of $ann_{A'}(u)$. Then
$w:=\frac{1}2(u+v)\in A'$. Moreover, $w$ has norm $\frac{1}4(2+6)=2$ and $\la
u, -w\ra=- 1$. Thus $A'\cong A_2$ and $A\cong AA_2$.
\end{proof}

\begin{nota}\labtt{Q'R'2}
Let $Q'\cong \dih{12}{16}$.  Let
$M$ and $N$ be $EE_8$ sublattices of $Q'$ such that $Q'=M+N$. In this case,  the
$EE_8$-involutions $t_M$, $t_N$ generate a dihedral group of order $12$.
As in Notation \refpp{dih12}, we set $h=t_Mt_N$, $g= h^2$ and $u= h^3$.
Then $\la u\ra$ is the center of $\la t_M, t_N\ra$. Let $E$ be the $(-1)$-eigenlattice of $u$ in
$Q'$.
Then $E$ is also an $EE_8$-sublattice of $Q'$ and $E\cap M \cong E\cap N\cong DD_4$ (see
\cite{GL}).

Set $F_M:=ann_M(N)$ and $F_N:= ann_N(M)$. Then by \cite{GL}, we have
$F_M\cong F_N\cong AA_2$ and $J:= ann_{Q'}( F_M\perp F_N)\cong K_{12}$. In
addition, $E\cap J\cong AA_2\perp AA_2$.
\end{nota}

\begin{lem}
 Let $Y$ be an $EE_8$ sublattice of $Q'$. Then either $Y=E$ or $E\cap Y \cong DD_4$.
\end{lem}

\begin{proof}
Suppose $Y\neq E$. Then by the classification of $EE_8$-pairs \cite{GL}, we have $E\cap Y \cong 0,
AA_1$, $AA_1^2$, $AA_2$ or $DD_4$.
\medskip

\noindent \textbf{Case 1:} $E\cap Y=0$. In this case, $E+Y$ is a full rank sublattice of $Q'$. By
the classification of $EE_8$-pairs and $\dg{Q'}=6^4$, the only possible case is $E+Y= Q'$. Then we
have $F_E:=ann_{E}(Y)\cong AA_2$, $F_Y:=ann_Y(E)\cong AA_2$ and $ann_{E}(F_E)\cong EE_6$.
Note also that $F_Y\subset ann_{Q'}(E)\cong R'$ and $ann_{R'}(AA_2)\cong AA_2^3$. Thus, we obtain a
full sublattice of type $EE_6\perp AA_2^3$ in $ann_{Q'} (F_E\perp F_Y)\cong K_{12}$. It is
impossible since the 3-rank of $\dg{EE_6\perp AA_2^3}$ is $4$ and $\dg{K_{12}}=3^6$.

\medskip

\noindent \textbf{Case 2:} $E\cap Y\cong AA_1$. Then $E+Y\cong \dih{8}{15}$.
In this case, the $EE_8$ involutions $t_E, t_Y$ generate a dihedral group of
order $8$. Let $z=(t_Et_Y)^2$ and let $Z$ be the $(-1)$-eigenlattice of $z$ in
$Q'$. Let $F_E$ and $F_Y$ be the fixed point sublattices  of $z$ in $E$ and $Y$,
respectively. By the analysis in \cite{GL}, we know that $F_E\cap F_Y\cong
AA_1$, $F_E\cong F_Y\cong DD_4$ and $Z\cong EE_8$.  Then $ann_{E+Y}(
F_E\perp Z) \supset ann_{F_Y}(F_E\cap F_Y)\cong AA_1^3$. However,
$ann_{Q'}(F_E\perp Z) \cong \sqrt{6}D_4$  has no norm 4 vectors. It is a
contradiction.

\medskip

\noindent \textbf{Case 3:} $E\cap Y\cong AA_1^2$. Then $R' \supset ann_Y(E) \cong DD_6$ but $R'
\cong A_2\otimes D_4$ doesn't have such a sublattice.

Recall that the norm $4$ vectors of $R'$ have the form
$ \a\otimes \b$, where $\a, \b$ are roots of $A_2$ and $D_4$, respectively.
Suppose $R'$ contains a sublattice isometric to $DD_6$. Since $DD_6\supset AA_1^6$, there exist
roots $\a_1, \dots, \a_6 \in A_2$ and $\b_1, \dots, \b_6\in D_4$ such that
\[
\la \a_i\otimes \b_i, \a_j\otimes \b_j\ra =  \la \a_i, \a_j\ra \cdot \la \b_i, \b_j\ra = 4\delta_{i,j},
\]
for any $i,j=1, \dots, 6$.  Since $\la \a_i, \a_j\ra \neq 0 $ for any roots $\a_i, \a_j\in A_2$, we must have
$ \la \b_i, \b_j\ra=2 \delta_{i,j}$ for $i,j=1, \dots, 6$. It is impossible because $D_4$ has rank $4$.

\medskip

\noindent \textbf{Case 4:} $E\cap Y\cong AA_2$. Then $E+Y\cong \dih{6}{14}$
and $ann_{E+Y}(E\cap Y)\cong K_{12}$. By the analysis in \cite{GL}, we have
\begin{equation}\labtt{aa2}
ann_{E}(E\cap Y)\cong EE_6 \quad \text{  and } \quad (E+Y)\cap R'=ann_{E+Y}(E)\cong
\sqrt{6}E_6^*.
\end{equation}

Let $A=ann_{Q'}(E+Y)$. Then $\dg{A}\cong 2^2\cdot 3$. Then $A\cong
AA_2$ by Lemma \ref{A}. Then  $ann_{R'}(A)\cong AA_2^3$. However,
$ann_{R'}(A)= (E+Y)\cap R' \cong \sqrt{6}E_6^*$ by \refpp{aa2}. It is  a
contradiction.

Therefore, the only possible case is $E\cap Y\cong DD_4$. Note that such a case
occurs since $E\cap M\cong E\cap N \cong DD_4$ (see \cite{GL}).
\end{proof}

\begin{prop}\labtt{EE8inQ'}
We use the same notation as in Notation \refpp{Q'R'}. Let $Y$ be an $EE_8$-sublattice
in $Q'$. Then $Y=E$, $g^i(M)$ or $g^i(N)$, where $g=(t_Mt_N)^2$ and $i=0,1,2$.
\end{prop}

\begin{proof}
 Let $Y$ be an $EE_8$ sublattice in $Q'$ and $E\neq Y$. Then $E\cap Y\cong DD_4$ and  $E+Y \cong
\dih{4}{12}$. Moreover, we have $Y\cap R'= ann_{Y}(Y\cap E) \cong DD_4$. Since there are only 3
$DD_4$-sublattices in $R'$, we must have $Y\cap R'= g^i(M) \cap R'$ for some $i=0,1,2$.  It now
follows $E+Y= E+g^i(M)$ and we have the desired conclusion since $\dih{4}{12}$ has only 3 $EE_8$
sublattices in it.
\end{proof}

\begin{prop}\labtt{EE8inQ}
Let $Q\cong \dih{6}{14}$. Then $Q$ contains exactly three $EE_8$-sublattices.
\end{prop}

\begin{proof}
Let $M$ and $g$ be defined as in Notation \ref{Q'R'2}. Then $M+gM\cong
\dih{6}{14}$. Thus we can view $Q\cong \dih{6}{14}$ as a sublattice of $Q'$.  The
result now follows from Proposition \refpp{EE8inQ'}.
\end{proof}

\section{The containment $C_{O(\L )}(\la t_M, t_N \ra)$ in $C_{O(\L )}(t_E)$}

The group $C_{O(\L )}(t_E)$ has the form $2.\brw 4 \cong (2\times 2^{1+8})\Omega^+(8,2)$.   The groups $\brw d$ may be analyzed by the methods of \cite{gal}.  We give some discussion of $C_{O(\L )}(t_E)$ and $C_{O(\L )}(\la t_M, t_N \ra)$ from the BRW-viewpoint.

\begin{lem}\labtt{ginc(z)}
Let $u$ and $g$ be defined as in \refpp{Q'R'}.
The element $g\in C_{O(\L )}(u)$ acts on the Frattini quotient of $O_2(C_{O(\L )}(u))$ with 4-dimensional fixed points.  In the notation of \cite{G12}, $g$ is in the class of $\th_2$.
\end{lem}
\pf
An element of order 3 in the class of $\th_i$ acts on $Fix(u)$ with trace $-8, 4, -2, 1$ for $i=1,2,3,4$, respectively.

By \refpp{dih12}, the sum of the traces in $\la g, u\ra$ is $48$.  By the orthogonality relations, $rank(Fix(g)\cap \L^+(u))=8$ whence
$rank(Fix(g)\cap \L^-(u))=4$.   This implies that the trace of $g$ on $\L^+(u)$ is  $8-\half (16-8)=4$, whence $i=2$, as claimed.
\eop

\begin{lem}\labtt{tinc(z)}
Let $t\in \{t_M, t_N\}$.   Then $t$ acts with trace 8 on $\L^+(u)$ and trace 0 on $\L^-(u)$.
\end{lem}
\pf
Define $U:=O_2( C_{O(\L )}(u) )$.
The trace of $t$ on $\L^+(u)$ is 0 if and only if $t$ is conjugate in $C_{O(\L )}(u)$ to $tu$ and otherwise the trace has the form $\pm 2^d$ where $2d$ is the dimension of the fixed points of $t$ in its action on $U/U'$.

We claim that $t$ does not have trace 0.  Supposing otherwise, we see that $t$ has trace 8 on $\L^-(u)$, which means that $t\in U$.  This is impossible since $t$ inverts $g$.   The claim follows.

Suppose $t=t_M$.   The rank of $M\cap E $ is 4  by \refpp{gmn}. Therefore,
$\L^+(u)\cap M$ has rank 4.    This leads to trace $12-4=8$ on $\L^+(u)$. Trace 0
on $\L^-(u)=E$ follows. \eop

\begin{lem}\labtt{WF4}
$Weyl(F_4)\cong O(D_4)\cong O^+(4,3)$.
\end{lem}
\pf
The group $Weyl(F_4)$ acts faithfully as isometries on the $D_4$-lattice, which is spanned by the long roots of $F_4$.
This lattice has determinant 4, whence its reduction modulo 3 is nonsingular.
The $\FF_3$-valued form is split since the sublattice of type $A_1^4$ maps onto $D_4/3D_4$ (take an orthogonal set of roots $p,q,r,s$ and note that
$p+q+r, q-r+s$ generates a direct summand of $\ZZ p+\ZZ q+\ZZ r+\ZZ s$ and that their image in
$D_4/3D_4$ is totally singular).
We therefore get a homomorphism $Weyl(F_4)\rightarrow O^+(4,3)$.  Both groups have order $2^73^2$ and
this map is monic because any normal subgroup of $Weyl(F_4)$ contains the central involution which acts as $-1$ on $D_4/3D_4$.
\eop

\begin{nota}\labtt{contain1}
Let $G:=\brw 4$,
$U:=O_2(G)\cong 2^{1+8}_+$, $\tau \in G$ an involution of trace 8 on the natural module (dimension 16) and $\th \in G$ an element of order 3 such that
$[U, \th ]/U'$ has rank 4.
Define $G_1:=C_G([U,\th ])$, $G_2:=C_G(G_1)$.   Then $G_1\cong G_2\cong \brw 2$.
\end{nota}

\begin{lem}\labtt{contain2}
$[U,\tau ]\le [U,\th]$.
\end{lem}
\pf
Since the trace of $\tau$ is 8, $[U,\tau]/U'$ has rank 2.    Since $\tau$ inverts $\th$ by conjugation and leaves invariant $[U,\th]$, the containment follows.
\eop

\begin{coro}\labtt{contain3}
$\tau \in G_2$.
\end{coro}
\pf
Note that the direct product decomposition $[U,\th ]/U' \times C_U(\th)/U'$ is orthogonal in the sense of the natural quadratic form on $U/U'$.   The stabilizer of one summand stabilizes both and that stabilizer is $G_1 G_2$, a central product.
So $\tau \in G_1G_2$ and since $\tau$ stabilizes both summands and has commutator rank 2 on the quadratic space, $\tau \in UG_2$.   If $\tau \notin G_2$, there exists $x\in (U\cap G_1) \cap G_2 \setminus U'$ so that $\tau \in xG_2$.   There exists $y \in U\cap G_1$ so that $[x,y]$ generates $U'$.    Then $[y,\tau]$ generates $U'$.   This means that $\tau $ has trace 0 on the natural module, a contradiction.
\eop

The next two results apply to $\th, \tau \in G_2\cong Weyl(F_4)$ acting on the $D_4$-lattice.

\begin{lem}\labtt{contain4}
In $O(D_4)$, if $P\cong 3^2$ is a Sylow 3-group and $H$ is its normalizer,
$H= Z(O_2(O(D_4))) \times H_1\times H_2$,
for a unique pair of dihedral subgroups $H_1, H_2$
generated by reflections.
Furthermore, an element of order 3 in an
$H_i$ has trivial fixed points on the Frattini quotient of $O_2(O(D_4))$.
Elements of order 3 in $P\setminus (H_1\cup H_2)$ have rank 2
fixed points on the $D_4$-lattice.
\end{lem}
\pf
We observe that $Weyl(F_4)$ contains a natural $Weyl(A_2)\times Weyl(A_2)$ generated by reflections (at roots of different lengths).   Since $P$, a Sylow 3-group of this (also a Sylow 3-group of $O(D_4)$), acts without fixed points on the Frattini factor of $O_2(O(D_4))$, $H\cong Z(O_2(O(D_4))) \times Dih_6 \times Dih_6$.
The subgroup of this generated by roots is the group $Weyl(A_2)\times Weyl(A_2)$ mentioned above.

For the second, let $x$ be  an element of order 3 in $H_1$.
Since $x$ has trace 1 on the lattice,  $x$ has trivial fixed points on  the Frattini quotient of $O_2(O(D_4))$.   The remaining four elements of order 3 in $P$ have rank 2 fixed points on the Frattini factor and
0 fixed point sublattice.

Under conjugacy by $H$, the elements of order 3 are partitioned into orbits $\{ x, x^{-1}\}$, $\{ y, y^{-1} \}$ and the remaining set of four elements of order 3 (the ones which have nontrivial fixed points on  the Frattini quotient of $O_2(O(D_4))$.
The final statement follows.
\eop

\begin{coro} \labtt{contain4.5} We use the notation of \refpp{contain4}.
Let  $\{i,j\}=\{1,2\}$.
If $\th \in H_i$, then
$C_{O(D_4)}(\th ) = Z(O_2(O(D_4))) \times \la \th \ra  \times H_j$.
Also,
$\tau \in H_i$ and
$C_{O(D_4)}(\la \th, \tau \ra) = Z(O_2(O(D_4)))H_j$.
\end{coro}
\pf
Since $\th$ centralizes just $Z(O_2(O(D_4)))$ on $O_2(O(D_4))$, the form of $C(\th)$ follows from \refpp{contain4}.
An element of $H$ which inverts $P$ has trace 0 on the lattice since it induces outer automorphisms on each quaternion group in $O_2(O(D_4))$.    Therefore,
$\tau \in Z(O_2(O(D_4)))r$, where $r\in H_i$ is a reflection.   Since  $r$ has trace 2 and since the central involution acts as $-1$ on the lattice, $\tau$ has trace $\pm 2$.
Since $\tau$ has trace 8 on the rank 16 representation \refpp{contain1} and this module is a tensor product for the central product $G_1G_2$, it follows that in the action of $G_2$ on $D_4$, $\tau$ has trace 2 and $\tau$  is in fact a reflection.
 \eop

\begin{nota}\labtt{contain5}
Let $t_M, t_N, t_E, h, g$ be as in Section (on $DIH_{12}(16)$).
Then $u:=t_E$ is the central involution of $\la t_M, t_N \ra$.

Let $\psi$ be a homomorphism of $C_{O(\L )}(u)$ onto $\brw 4$.
Denote $\tau :=\psi (t_M), \th :=\psi (g)$.
 Note that
\refpp{ginc(z)}, \refpp{tinc(z)} imply that Notation \refpp{contain1} applies to $\tau, \th$.

The centralizer  in $G:=\brw 4$ of $\la \tau, \th \ra$ has been discussed in
\refpp{contain4.5}.
Note that
$Ker(\psi )=\la -u \ra$ and that $\psi$ maps $C_{O(\L )}(\la t_M, t_N \ra)$ to a
subgroup of $C_G( \la \tau, \th \ra )\cong O^+(4,3)\times Dih_6$. According to
\refpp{CD1}, $C_{O(\L )}(\la
t_M, t_N \ra)$ has order $2^83^3$, the same as $O^+(4,3)\times Dih_6$.  This
means that $\psi (C_{O(\L )}(\la t_M, t_N \ra))$ has index 2 in $C_G( \la \tau, \th
\ra )\cong O^+(4,3)\times Dih_6$.
\end{nota}

\section{Trivial action on lattices mod 2 }

\begin{lem}
Suppose that the involution $t$ acts on the  abelian group $L$ which has no 2-torsion.
 Assume that $t$ is trivial on $L/2L$.   Then $L$ is the direct sum of eigenlattices for $t$.
\end{lem}
\pf There exists an endomorphism $a$ of $L$ so that $t=1-2a$.  Then
$1=t^2=(1-2a)^2=1-4a+4a^2=1+4(a^2-a)$ and absence of 2-torsion imply that
$0=a^2-a=a(a-1)$.

For $x\in L$, $x=ax + (1-a)x$, whence $L$ is the sum of subgroups $Ker(1-a)$ and
$Ker(a)$.

Let $x\in Ker(a)$.  Then $tx=(1-2a)x=x$.   If $y\in Ker(1-a)$,
$ty=(1-2a)y=(-1+2(1-a))y=-y$.  Therefore $L$ is the sum of the 1-eigenlattice and
$(-1)$-eigenlattice of $t$.   Their intersection is 0 since $L$ is free of 2-torsion.
\eop

\begin{coro}\labtt{tL2L}
Suppose that the involution $t$  is an isometry of the orthogonally
indecomposable rational lattice $L$ such that $t$ acts trivially on $L/2L$.   Then
$t=1$ or $t=-1$.
\end{coro}
\pf By the Lemma, $L$ is a direct sum of eigenlattices for $t$.   Since $t$ is an
isometry of $L$, this is an orthogonal direct sum, whence indecomposability of
$L$ implies that one of the summands is 0. \eop

\begin{lem}\labtt{liftinvollatticemod2}  Let $L$ be a finite rank positive definite orthogonally indecomposable lattice.
Suppose that $u\in O(L)$  has the property that $u^2$ acts trivially on $L/2L$ and
$(u-1)(L/2L)$ has dimension less than $\frac 12 rank(L)$.   Then $u$ is an involution.
\end{lem}
\pf Since $L$ is an orthogonally indecomposable lattice, $u^2=\pm 1$ by
\refpp{tL2L}.    Suppose that $u^2=-1$.   Then $L$ is a torsion free module for
$\ZZ [u]\cong \ZZ [\sqrt {-1}]$, a PID.  Therefore,  $L$ is a free module and
$L/(u-1)L$ has $\FF_2$-dimension exactly $\half rank(L)$, a contradiction. \eop

\section{Facts about discriminant groups}

\begin{lem}[Lemma A.9 of \cite{gl5A}]\labtt{N*}
Let $X$, $Y$ be sublattices of the lattice $L$ where $Y$ is a direct summand of $L$, $L=X+Y$ and
$(det(Y), det(L))=1$.  Then the  natural map of $X$ to the discriminant group $\dg Y=Y^*/Y$ is onto.
\end{lem}

\begin{lem}\labtt{M+N}
Let $X$ and $Y$ be sublattices of the unimodular lattice $L$ such that $X\cap Y=0$ and $X+Y$ is a direct summand of $L$. Then for any $y\in Y^*$, there exists $\b\in L$ such that
(1)  $P_{Y}(\b)=y \mod 2Y $ and  (2) $\la P_{X}(\b), X\le   2\ZZ$.
\end{lem}

\pf Since $X\cap Y=0$, Conditions (1) and (2) define uniquely a $\ZZ$-linear map $\varphi: X+Y \to \ZZ_2$ such that $\varphi(\a)= \la y, \a\ra \mod 2$ for $\a \in Y$ and $\varphi(\a)=0$ for $\a \in X$. Moreover, $X+Y$ is a direct summand of $L$ and thus the natural map from $L$ to $(X+Y)^*$ is a surjection. Hence there is a $\b\in L$ such that $\varphi(\a)= \la \b, \a\ra \mod 2$ for all $\a\in X+Y$ and $\b$ satisfies Conditions (1) and (2). \qed

\end{document}